\documentclass[reqno]{amsart}

\newif\ifdraft

\iffalse 

  \usepackage[landscape,includehead,headheight=12pt,headsep=10pt,
  left=\columnsep,right=\columnsep,top=12.98999pt,bottom=\columnsep,]{geometry}
  \twocolumn
\else
\fi

\usepackage{latexsym,bbm,amssymb}
\usepackage{multirow,caption}
\usepackage{hyperref}
\hypersetup{pdfstartview=}

\newtheorem{thm}{Theorem}[section]
\newtheorem{cor}[thm]{Corollary}
\newtheorem{lem}[thm]{Lemma}
\newtheorem{prop}[thm]{Proposition}
\theoremstyle{definition}

\theoremstyle{remark}
\newtheorem{rem}{Remark}[section]

\numberwithin{equation}{section}

\newcommand{\R}{\mathbb{R}}
\newcommand{\nwc}{\newcommand}
\nwc{\sgn}{\mathop{\rm sgn}\nolimits}
\nwc{\supp}{\mathop{\rm supp}\nolimits}
\nwc{\qref}[1]{(\ref{#1})}
\nwc{\ip}[1]{\langle{#1}\rangle}
\nwc{\D}{\partial}
\nwc{\Dt}{\partial_t}
\nwc{\Dx}{\partial_x}
\nwc{\Dxx}{\partial_x^2}
\nwc{\intr}{\int_0^\infty}
\nwc{\inti}{\int_0^1}
\nwc{\inv}{^{-1}}
\nwc{\one}{\mathbbm{1}}
\nwc{\ibvp}{ \qref{eq:0}, \qref{eq:ibc}}
\nwc{\init}{^{\rm in}}
\nwc{\loc}{_{\rm loc}}
\nwc{\fin}{f\init}
\nwc{\bfin}{\bar f\init}
\nwc{\nin}{n\init}
\nwc{\mnin}{m\init}
\nwc{\bnin}{{\bar n}\init}
\nwc{\fep}{f_{\epsilon}}
\nwc{\nep}{n_{\epsilon}}
\nwc{\mep}{m_{\epsilon}}
\nwc{\nmu}{n_\mu}
\nwc{\eps}{\epsilon}
\nwc{\epb}{{\bar\epsilon}}
\nwc{\ap}{\alpha}

\nwc{\Jdif}{{\tilde J}}
\nwc{\QET}{Q_{\eps,T}}
\nwc{\Qplus}{Q} 
\nwc{\Bplus}{B_+}
\nwc{\dX}{} 
\nwc{\llocq}{L^2\loc(Q)}

\nwc{\neph}{n_h}
\nwc{\holder}{H\"older}

\nwc{\hide}[1]{}  
\nwc{\myline}{\bigskip  \hrule  \bigskip}

\def\Xint#1{\mathchoice
  {\XXint\displaystyle\textstyle{#1}}%
  {\XXint\textstyle\scriptstyle{#1}}%
  {\XXint\scriptstyle\scriptscriptstyle{#1}}%
  {\XXint\scriptscriptstyle\scriptscriptstyle{#1}}%
\!\int}
\def\XXint#1#2#3{ \setbox0=\hbox{$#1{#2#3}{\int}$ }
\vcenter{\hbox{$#2#3$ }}\kern-.6\wd0}
\def\avint{\Xint-} 

\makeatletter
\def\serieslogo@{}
\makeatother \makeatletter
\def\@setcopyright{}
\makeatother

\begin{document}

\title[Global Dynamics of Bose-Einstein Condensation]
      {Global Dynamics of Bose-Einstein Condensation \\
       for a Model of the Kompaneets Equation}

\author{C. David Levermore}
\address{
    Department of Mathematics,
    University of Maryland,
    College Park, MD 20742}
    \email{lvrmr@math.umd.edu}

\author{Hailiang Liu}
\address{
    Department of Mathematics,
    Iowa State University,
    Ames, IA 50011}
    \email{hliu@iastate.edu}

\author{Robert L. Pego}
\address{
    Department of Mathematical Sciences,
    Carnegie Mellon University,
    Pittsburgh, PA 15213.}
    \email{rpego@cmu.edu}

\date{\today}




\begin{abstract} 

   The Kompaneets equation describes a field of photons exchanging 
energy by Compton scattering with the free electrons of a homogeneous, 
isotropic, non-relativistic, thermal plasma.  This paper strives to 
advance our understanding of how this equation captures the phenomenon 
of Bose-Einstein condensation through the study of a model equation.  
For this model we prove existence and uniqueness theorems for global 
weak solutions.  In some cases a Bose-Einstein condensate 
will form in finite time, and we show that it will 
continue to gain photons forever afterwards.  Moreover we show that every 
solution approaches a stationary solution for large time.  Key tools 
include a universal super solution, a one-sided Oleinik type 
inequality, and an $L^1$ contraction.

\end{abstract}

\maketitle






\section{Introduction}

   Photons can play a major role in the transport of energy in a fully 
ionized plasma through the processes of emission, absorption, and 
scattering.  At high temperatures or low densities, the dominant 
process can be Compton scattering off free electrons.  We make the 
simplification that the plasma is spatially uniform, isotropic, 
nonrelativistic, and thermal at temperature $T$.  We also neglect the 
heat capacity of the photons and assume that $T$ is fixed.  If the 
photon field is also spatially uniform and isotropic then it can be 
described by a nonnegative number density $f(x,t)$ over the unitless 
photon energy variable $x\in(0,\infty)$ given by
$$
  x = \frac{\hbar |k| c}{k_B T},
$$
where $\hbar$ is Planck's constant, $c$ is the speed of light, $k_B$ 
is Boltzmann's constant, and $k$ is the photon wave vector.  Because 
$x$ is a unitless radial variable, the total photon number and 
(unitless) total photon energy associated with $f(x,t)$ are then 
given by
$$ 
  N[f] = \int_0^\infty f \, x^2 \, dx \,, \qquad
  E[f] = \int_0^\infty f \, x^3 \, dx \,.
$$
When the only energy exchange mechanism is Compton scattering of the 
photons by the free electrons in the plasma then the evolution of $f$ 
is governed by the {\sl Kompaneets equation} \cite{K57} 
\begin{equation}
  \label{eq:ko}
  \partial_t f 
  = \frac{1}{x^2} \, 
    \partial_x \left[ x^4 \big( \partial_x f + f + f^2 \big) \right] \,.
\end{equation}
This Fokker-Planck approximation to a quantum Boltzmann equation is 
justified physically by arguing that little energy is exchanged by 
each photon-electron collision.  

   Because $x$ is a radial variable, the associated divergence 
operator has the form $x^{-2}\partial_x x^2$.  Thereby we see from 
\qref{eq:ko} that the diffusion coefficient in the Kompaneets equation 
is $x^2$, which vanishes at $x=0$.  This singular behavior allows the 
$f^2$ convection term to drive the creation of a photon concentration 
at $x=0$.  This hyperbolic mechanism models the phenomenon of 
Bose-Einstein condensation.  Our goal is to better understand how the 
Kompaneets equation generally describes the process of relaxation to 
equilibrium over large time and how it captures the phenomenon of 
Bose-Einstein condensation in particular.

   Rather than addressing these questions for the Kompaneets equation
\qref{eq:ko} we will consider the model Fokker-Planck equation
\begin{equation}
  \label{eq:0}
  \partial_t f 
  = \frac{1}{x^2} \,
    \partial_x \left[ x^4 \big( \partial_x f + f^2 \big) \right] \,,
\end{equation}
posed over $x\in(0,1)$ and subject to a zero flux boundary condition 
at $x=1$.  This model is obtained by simply dropping the $f$ term that 
appears in the flux of the Kompaneets equation \qref{eq:ko} and 
reducing the $x$-domain to $(0,1)$.  As we will see, this model shares 
many structural features with the Kompaneets equation.  In particular, 
it shares the $x^2$ diffusion coefficient and the $f^2$ convection 
term that allow the onset of Bose-Einstein condensation.  The neglect 
of the $f$ term in the flux of the Kompaneets equation is a reasonable 
approximation during the onset of Bose-Einstein condensation when we 
expect $f$ to be large.  The advantage of model \qref{eq:0} is that 
we know some estimates for it that have no known analogs for 
\qref{eq:ko}, and which facilitate the study of 
condensate dynamics and equilibration. 
A disadvantage of \eqref{eq:0} is that its equilibrium solutions
differ from those of \eqref{eq:ko}, so we may expect
the long-time behavior of its solutions to be 
similar to that of solutions of \eqref{eq:ko}
only in a qualitative sense.


\subsection{Structure of the Kompaneets Equation}

  Here we describe some structural features of the Kompaneets equation
that will be shared by our model.  First, solutions of \qref{eq:ko} 
formally conserve total photon number $N[f]$.  Indeed, we formally 
compute that
\[
  \frac{d}{dt} N[f] 
  = x^4 \big( \Dx f + f + f^2 \big) \Big|_0^\infty = 0 \,,
\]
under the expectation that the flux vanishes as $x$ approaches $0$ and 
$\infty$.  Second, solutions of \qref{eq:ko} formally dissipate 
quantum entropy $H[f]$ given by
\[
  H[f] = \int_0^\infty \! h(f,x) \, x^2 dx \,, \qquad 
  h(f,x) = f \log(f) - (1 + f) \log(1 + f) + x f \,.
\]
Indeed, because
\begin{align*}
  h_f(f,x) & = \log(f) - \log(1 + f) + x
             = \log\left( \frac{e^x f}{1 + f} \right) \,,
\\
  \partial_x h_f & = h_{ff} \partial_x f + 1 
                   = \frac{1}{f (1 + f)} (\partial_x f + f + f^2) \,,
\end{align*}
we formally compute that
\[
  \frac{d}{dt} H[f] = \intr  h_f(f,x) \, (\Dt f) \, x^2 dx
  = - \intr x^4 f (1 + f) \big( \Dx h_f(f,x) \big)^2 dx \leq 0 \,.
\]
By this ``$H$ theorem,'' we expect solutions to approach an equilibrium  
for which $\Dx h_f(f,x)=0$.  These equilibria have the Bose-Einstein 
form
$$
  f = f_\mu(x) = \frac{1}{e^{x + \mu} - 1} \,, \qquad 
  \text{for some $\mu \geq 0$} \,.
$$

   At this point a paradox arises.  The total photon number for the 
equilibrium $f_\mu$ is
\[
  N[f_\mu] = \int_0^\infty \frac{x^2}{e^{x + \mu} - 1} \, dx \,.
\]
This is a decreasing function of $\mu$ over $[0,\infty)$, and is 
thereby bounded above by $N[f_0]$, which is finite.  Because total 
number is supposed to be conserved, we expect any solution to relax to an equilibrium 
$f_\mu$ with the same total number as the initial data, satisfying 
$N[f_\mu] = N[\fin]$.  But if the initial number $N[\fin]>N[f_0]$ then 
no such equilibrium exists!


\subsection{Bose-Einstein Condensation}

   The foregoing paradox indicates that there must be a breakdown in 
the expectations given above.  Previous studies (\cite{EHV98} in particular)
have shown that a breakdown in the no-flux condition at $x=0$ can occur.  
A physical interpretation of a nonzero photon flux at $x=0$ is that the photon 
distribution forms a concentration of photons at zero energy (i.e.,
energy that is negligible on the scales described by the model).  
This Bose-Einstein condensate accounts for some of the total photon number.  
See especially the works \cite{LZ69,CL86,EHV98,JPR06}, and the discussion of related literature
in subsection~\ref{s:discuss} below.  As massless, 
chargeless particles of integer spin, photons are the simplest bosons.  
Indeed, S.~N.~Bose had photons in mind in 1924 when he proposed his new 
way of counting indistinguishable particles, work soon followed by 
Einstein's prediction of the existence of the condensate.  Yet it was 
not until 2010 that the first observation of a photon condensate was 
reported by Martin Weitz and colleagues \cite{KSVW10}.

   In the present context, we can gain insight into this phenomenon by 
dropping the diffusion term in \qref{eq:ko}, as discussed by Levich 
and Zel'dovich \cite{LZ69}.  In this case the Kompaneets equation 
simplifies to the first-order hyperbolic equation
$$
  \partial_t f 
  = \frac{1}{x^2} \,
    \partial_x \left[ x^4 \big( f + f^2 \big) \right] \,.
$$
Letting $n = x^2f$, this becomes 
\begin{equation}
  \label{eq:koon}
  \partial_t n = \partial_x \left[ x^2 n + n^2 \right] \,,
\end{equation}
whose characteristic equations are
$$
  \dot x = - x^2 - 2 n \,, \qquad  \dot n = 2 x n \,.
$$
Because $n \geq 0$, the origin $x=0$ is an outflow boundary,
and no boundary condition can be specified there.  Clearly any 
nonzero entropy solution will develop a nonzero flux of photons 
into the origin in finite time, leading to the formation of a 
condensate.

   The fact the $f^2$ convection term plays an essential role in the 
formation of Bose-Einstein condensates is illustrated by considering 
what happens when that term is dropped from the Kompaneets equation 
\qref{eq:ko}.  This leads to the linear degenerate parabolic equation
\begin{equation}
  \label{eq:klin}
  \partial_t f 
  = \frac{1}{x^2} \, 
    \partial_x \left[ x^4 \big( \partial_x f + f \big) \right] \,.
\end{equation}
This equation is the analog of the Kompaneets equation for classical 
statistics.  Its solutions formally conserve $N[f]$ and dissipate the 
associated entropy
$$
  H[f] = \int_0^\infty h(f,x) \, x^2 dx \,, \qquad \text{where} \quad
  h(f,x) = f \log(f) - f + x f \,.
$$
Its family of equilibria is
$$
  f_\mu(x) = e^{- x - \mu} \,,  \qquad \text{for some $\mu \in \R$} \,.
$$
The initial-value problem for \qref{eq:klin} is well-posed in cones of 
nonnegative densities $f$ such that 
$$
  \int_0^\infty (e^x f)^p e^{-x} x^2 dx < \infty \,, \qquad   
  \text{for some $p\in(1,\infty)$} \,.
$$
These solutions \cite{KL90}

\begin{itemize}

\item are smooth over $\R^+\times \R^+$,

\item are positive over $\R^+\times \R^+$ provided that $f\init$ is 
      nonzero,

\item satisfy all the expected boundary conditions,

\item conserve $N[f]$ and dissipate $H[f]$ as expected,

\item approach $f_\mu$ as $t\to\infty$, where $N[f\init] = N[f_\mu]$.

\end{itemize}

\noindent
In particular, the no-flux boundary condition is satisfied at $x=0$ 
without being imposed!  Therefore, no Bose-Einstein concentration 
happens!  


\subsection{Present Investigation}

   Of course, solutions of the hyperbolic model \qref{eq:koon} may 
develop shocks at any location.  However, the diffusion term in the
Kompaneets equation \qref{eq:ko} prevents shock formation for $x>0$.
The results of Escobedo et al.~\cite{EHV98} prove that the degeneracy 
of its diffusion does {\sl not} prevent shock formation at $x=0$.  
These authors proved that there exist solutions of \qref{eq:ko} that 
are regular and satisfy no-flux conditions for $t$ on a bounded 
interval $0<t<T_c$ (which is solution-dependent), but at time $t=T_c$ 
the flux at $x=0$ becomes nonzero.  Such solutions exist for 
arbitrarily small initial photon number.  Moreover, global existence 
and uniqueness of solutions of \qref{eq:ko} was proved subject to a 
boundedness condition for $x^2 f$ for $x\in[0,1]$.  

   A number of interesting questions about solutions to the Kompaneets 
equation remain unanswered by previous studies: What happens to a 
condensate once it forms?  Can it lose photons as well as gain them?
Are there any boundary conditions at all that we can impose near $x=0$
that yield different condensate dynamics, allowing the condensate to 
interact with other photons?  Can we identify the long-time limit of 
any initial density of photons?

   In order to focus clearly on these questions, we have found it 
convenient to drop the linear term $x^2 f$ from the Kompaneets flux 
and consider the model equation \qref{eq:0}, which retains the 
essential features of nonlinearity and degenerate diffusion.  The 
equilibria of equation \qref{eq:0} are
\begin{equation}
  \label{eq:s0}
  f_\mu(x) = \frac{1}{x + \mu} \,, \qquad 
  \text{for some $\mu\geq0$} \,.
\end{equation}
We include these solutions in the class of functions considered by
restricting our attention to the interval $0<x<1$ and imposing a 
no-flux boundary condition at $x=1$. For these equilibria the
maximal total photon number is $N[f_0]=\frac12$.

   For this model problem, we shall assemble a fairly detailed 
description of well-posedness and long-time dynamics.  
We establish existence and 
uniqueness in a natural class of nonnegative weak solutions 
for initial data that simply has \textit{some} finite moment
\[
   \intr x^p \fin \, dx \,, \quad p\ge 2\,.
\]
These results are proved with 
essential use of estimates for hyperbolic (first-order)
equations, and establish that while the model Kompaneets
equation \eqref{eq:klin}  is parabolic for $x>0$,
the point $x=0$ remains an \textit{outflow} boundary at which no
boundary condition can be specified. 

The solution map is {\em nonexpansive} in $L^1$-norm with weight 
$x^2$.  Therefore the total photon number $N[f(t)]$ is nonincreasing 
in time.  A condensate can gain photons but never lose them, and
must form in finite time whenever $N[\fin]>N[f_0]$. 
Moreover, once it starts growing it never stops. 
Every solution relaxes to some equilibrium state $f_\mu$ 
in the long-time limit $t\to\infty$.  We cannot identify the limiting state in general, 
but the solution must approach the maximal steady state $f_0$
if the initial data $f_{\rm in}\ge f_0$ everywhere.

   The proofs of the results on long-time behavior are greatly 
facilitated by two features of the model problem \qref{eq:0}.
First, the problem admits a {\em universal super-solution}
$f_{\rm super}$ determined by
\begin{equation}
  x^2 f_{\rm super}(x,t) 
  = x + \frac{1 - x}{t} + \frac{2}{\sqrt{t}} \,.
\end{equation}
By consequence, for every solution, $x^2 f$ is in fact bounded in $x$ for each $t>0$, 
and moreover one has $\limsup_{t\to\infty}x^2f(x,t)\leq x=x^2f_0(x)$
for every solution, for example. Also, every solution satisfies
\[
  \Dx(x^2 f) \geq - \frac4t \,,
\]
which is Oleinik's inequality for admissible solutions of the 
conservation law \qref{eq:koon} after dropping the linear flux term 
$x^2n$.

\subsection{Literature on Related Problems}\label{s:discuss}
As indicated above, the Kompaneets equation is derived 
from a Boltzmann-Compton kinetic equation for photons interacting with a 
gas of electrons in thermal equilibrium---see \cite{K57}, and especially 
\cite{EMValle03} for a derivation and links to some of the physical literature.
Regarding the analysis of the Boltzmann-Compton equation itself, 
when a simplified regular and bounded kernel is adopted, 
Escobedo and Mischler \cite{EM01} studied 
the asymptotic behavior of the solutions, and showed that the photon 
distribution function may form a condensate at zero energy asymptotically in 
infinite time. 
Further, Escobedo et al. \cite{EMV04}  showed that the asymptotic behavior of solutions 
is sensitive not only to the total mass of the initial data but also to its 
precise behavior near the origin.  In some cases, solutions develop a Dirac 
mass at the origin for long times (in the limit $t\to\infty$)
in a self-similar manner.  For the Boltzmann-Compton
equation with  a physical kernel, some results concerning both global existence
and non-existence, depending on the size of initial data,
were obtained by Ferrari and Nouri \cite{FN06}.

A natural question is whether results analogous to those obtained in the present paper concerning the development 
of condensates may hold for other kinetic equations that govern boson gases, 
such as Boltzmann-Nordheim (aka Uehling-Uhlenbeck) quantum kinetic equations. 
Concerning these issues we refer to the work of 
H.~Spohn \cite{Sp10}, Xuguang Lu \cite{Lu11}, the recent analysis  
of blowup and condensation formation by 
Escobedo and Velazquez \cite{EV15}, and references cited therein.
Higher-order Fokker-Planck-type approximations to the 
Boltzmann-Nordheim equation were derived formally by Josserand et al \cite{JPR06}, 
and an analysis of the behavior of solutions has been performed recently by 
J\"ungel and Winkler \cite{JW15a,JW15b}.

Bose-Einstein equilibria and condensation phenomena also 
appear in classical Fokker-Planck models that incorporate 
a quantum-type exclusion principle \cite{KQ93,KQ94}. 
Concerning mathematical results on blowup 
and condensates for these models, we refer to work of Toscani~\cite{To12}
and Carrillo et al.~\cite{Car13} and references therein.

\subsection{Plan of the Paper}
In Section~\ref{s:results} we introduce our notion of weak solutions for \qref{eq:0} 
together with relevant notations, followed by precise statements of the 
main results, and a discussion of related literature.  
In Section~\ref{s:unique} we prove the uniqueness of weak solutions
for initial data with some finite moment. 
Existence is proved in Section~\ref{s:exist} by passing to the 
limit in a problem regularized by truncating the domain away from $x=0$.  

 In Section~\ref{s:bec} we establish that condensation 
must occur if the initial photon number $N[f_{\rm in}]>N[f_0]$, and we 
show that once a shock forms at $x=0$ in finite time, it will persist 
and continue growing for all later time.  
Large-time convergence to equilibrium is proved for 
every solution in Section~\ref{s:largetime}, using arguments related to
LaSalle's invariance principle. 

The paper concludes with three appendices that deal with 
several technical but less central issues. 
A simple, self-contained treatment of some anisotropic Sobolev embedding
estimates used in our analysis is contained in Appendix A.
The truncated problem used in Section~\ref{s:exist} requires 
a special treatment due to the fact that the zero-flux boundary condition 
at $x=1$ is nonlinear --- this treatment is carried out in Appendix B.  
A proof of interior regularity of the solution,
sufficient to provide a classical solution away from $x=0$ 
but up to the boundary $x=1$, is established in Appendix C.

{\ifdraft
\vfil\pagebreak 
\fi}

\section{Main results} 
\label{s:results}


\subsection{Model Initial-Value Problem}

   In light of the foregoing discussion, it is convenient to work 
with the densities
\begin{equation}
  \label{d:n}
  n = x^2 f \,, \qquad
  n\init = x^2 f\init \,.
\end{equation}
The flux in our model equation \qref{eq:0} can be expressed as
\begin{equation}
  \label{d:J}
  J = x^2 \partial_x n + n^2 - 2 x n \,.
\end{equation}
The initial-value problem for our model equation \qref{eq:0} that
we will consider is 
\begin{subequations}
  \label{ne-}
\begin{align}
  \partial_t n - \partial_x J & = 0 \,, && 0 < x < 1 \,, \ t > 0 \,,
\\
  J(1,t) & = 0 \,, && t > 0 \,, 
\\
  n(x,0) & = n\init(x), && 0 < x < 1 \,.
\end{align}
\end{subequations}
Here we have imposed the no-flux boundary condition at $x=1$, but do
not impose any boundary condition at $x=0$, where the diffusion 
coefficient $x^2$ vanishes.  

   We work with a weak formulation of the initial-value problem
\qref{ne-}.  We require the initial data $n\init$ to satisfy
\begin{equation}
  \label{init-data}
  n\init \geq 0 \,, \qquad  
  x^p n\init \in L^1((0,1]) \quad \text{for some $p\geq0$} \,.
\end{equation}
Let $Q=(0,1]\times(0,\infty)$.
We say $n$ is a {\em weak solution} of the initial-value problem 
\qref{ne-} if
\begin{subequations}
  \label{weak-soln} 
\begin{gather}
  \label{b}
  n \geq 0 \,, \qquad
  n \,, \  \D_x n \in L^2\loc(Q) \,, \quad 
\\
  \label{c:icp}
  x^p n \in L^1((0,1]\times(0,T)) \quad 
  \text{for every $T>0$} \,, 
\\
  \label{con-}
  n(\cdot,t) \to n\init \quad  
  \text{in $L^1\loc((0,1])$ as $t\to0^+$} \,,
\\
  \label{we-}
  \int_Q \bigl( n \,\Dt \psi - J \, \Dx\psi \bigr) \, dX = 0 
  \quad (dX = dx \, dt) \,,
\end{gather}
\end{subequations}
for every $C^1$ test function $\psi$ with compact support in $Q$. 
Condition \qref{b} is needed to make sense of the weak formulation
\qref{we-}.  Condition \qref{c:icp} is an admissibility condition we 
need to establish uniqueness.  Condition \qref{con-} gives the sense 
in which the initial data is recovered.


\subsection{Uniqueness, Existence, and Regularity}

   The following results establish the basic uniqueness, existence,
and regularity properties of weak solutions to (\ref{ne-}).   
Henceforth we will use $N[n]$ to denote the total photon number,
\[
N[n]=\int_0^1 n\,dx\,,
\] replacing the earlier notation $N[f]$.  We will 
also denote the positive part of a number $a$ by $a_+=\max\{a, 0\}$. 

\begin{thm}     [Stability and comparison] 
   \label{t.uniq}
Let $\nin$ and $\bnin$ satisfy \qref{init-data} for some $p\geq0$. 
Let $n$ and $\bar n $ be weak solutions of \qref{ne-} associated 
with the initial data $\nin$ and $\bnin$ respectively as defined by 
\qref{weak-soln}.  Set $c_p = p (p + 3)$.  Then
\begin{equation}
  \label{wf}
  \int_0^1 x^p (n - \bar n)_+ (x,t) \, dx 
  \leq e^{c_p t} \int_0^1 x^p (\nin - \bnin)_+ \, dx \,,  \quad
  \text{a.e.\ $t>0$} \,.
\end{equation}
Furthermore, if $\nin\ge\bnin$ a.e.\ on (0,1), then $n\ge\bar n$ 
a.e.\ on $Q$.  In particular, if $\nin=\bnin$ a.e.\ on (0,1)
then $n=\bar n$ a.e.\ on $Q$.
\end{thm}

From \qref{wf} we draw immediately the following conclusion on uniqueness.
\begin{cor}[Uniqueness] \label{c:uniq}
 Let $n$ and $\bar n $ be two weak solutions to (\ref{ne-}), subject to initial data
$\nin$, $\bnin$ respectively,  with $x^p\nin, x^p\bnin \in L^1((0, 1])$. Then
\begin{equation}\label{e.grest}
\int_0^1x^p|n(x, t)-\bar n(x, t) |\,dx \leq e^{c_pt} \int_0^1 x^p |\nin-\bnin|\,dx, \quad a.e.\; t>0.
\end{equation}
{For each initial data $\nin $ satisfying
$x^p n\init \in L^1(0,1)$ for some $p\ge0$,  
there exists at most one weak solution of (\ref{ne-}).
}
\end{cor}
\begin{rem} Because $c_{_0}=0$, if \qref{init-data} holds with $p=0$ 
then \qref{e.grest} is the $L^1$-contraction property
\begin{equation}
  \label{e:l1c}
  \| (n - \bar n)(t) \|_{L^1(0,1)} 
  \leq \| \nin - \bnin \|_{L^1(0, 1)} \,.
\end{equation}
In particular, the total photon number $N[n]$ is nonincreasing in time.
\end{rem}

\begin{thm} 
  [Existence and global bounds] 
  \label{t.exist}
  Let $\nin$ satisfy \qref{init-data} for some $p\geq0$.  Then there 
exists a unique  global weak solution $n$ of (\ref{ne-}) as defined by 
\qref{weak-soln}.  Moreover $x^p n\in C([0,\infty);L^1(0,1))$ and 
 we have the following bounds:
\begin{itemize}

\item[(i)] (A universal upper bound) For every $t>0$,
$$
  n \leq x + \frac{1-x}{t} + \frac{2}{\sqrt{t}} \qquad\mbox{for a.e.\ $x\in(0,1)$}\,,
$$
\item[(ii)] (Olenik-type inequality) For a.e.\ $(x,t)\in Q$,
$$
  \partial_x n \geq -\frac{4}{t} \,.
$$
\item[(iii)] (Energy estimate)  $n\in C((0,\infty),L^2(0,1))$, and whenever $0<s<t$, 
\begin{equation}
  \label{en-}
  \int_0^1 n^2(x,t) \, dx 
  + \int_s^t \int_0^1 [n^2 + x ^2 (\Dx n)^2] \, dx \, d\tau
  \leq \int_0^1 n^2(x,s) \, dx + \frac{8}{3} (t - s) \,.
  \end{equation}
\end{itemize}
\end{thm}
Note that the Oleinik-type inequality allows for the formation 
of `shock waves' in $n$ at $x=0$, but rules out oscillations.

\begin{thm}
  [Regularity away from $x=0$]
  \label{t:regeg}
  For the global weak solution $n$ from Theorem~\ref{t.exist},
  the quantities $n$, $\Dx n$, $\Dt n$ and $\Dxx n$ are
  locally \holder-continuous 
  on $Q$. Furthermore, $n$ is smooth in the interior of $Q$. 
\end{thm}


\subsection{Dynamics of solutions}

   Next we state our main results concerning the formation of 
condensates and the large-time behavior of solutions.  Observe that 
the bounds in (i) and (ii) of Theorem \ref{t.exist} imply the 
existence of the right limit $n(0^+,t)$, for each $t>0$.

\begin{thm}
  [Formation and growth of condensates]
  \label{t:cond}
  Let $\nin$ satisfy \qref{init-data} for some $p\geq0$. 
  Let $n$ be the unique global weak solution to (\ref{ne-}) associated 
  with $\nin$.  Then

\begin{itemize}

\item[(i)] (Conservation of photons.) For every $t>s>0$ we have
$$
  \int_0^1 n(x,t) \, dx 
  = \int_0^1 n(x,s) \, dx - \int_s^t n(0^+,\tau)^2 \, d\tau \,.
$$

\item[(ii)] (Persistence.)  There exists $t_*\in[0,\infty]$ such that
            $n(0^+,t)>0$ whenever $t> t^*$, and $n(0^+,t)=0$
            whenever $0\le t<t_*$.

\item[(iii)] (Formation.) If $N[\nin]>\frac{1}{2}$ then $n(0^+,t)>0$ whenever .
\begin{equation*} 
  \frac{1}{2\sqrt{t}} < \sqrt{1 + \delta} - 1 \,, \quad
  \text{where} \quad 2 \delta = N[\nin] - \frac{1}{2} \,.
\end{equation*}

\item[(iv)] (Absence.) If $\nin \leq x$, then $t_*=\infty$. I.e., for every $t>0$ we have $n(0^+,t)=0$ 
            and $N[n(\cdot,t)]=N[\nin]$.

\end{itemize}
\end{thm}

The formula in part (i) justifies a physical description of the photon energy
distribution that contains a Dirac delta mass at $x=0$, corresponding
to a condensate of photons at zero energy that keeps total photon number conserved.  By the formula in part (i), 
  the quantity
\[
  \int_s^t n(0^+,\tau)^2 \, d\tau
\]
is the number of photons that have entered the condensate between 
times $s$ and $t$.  
This quantity is nonnegative, meaning the condensate 
behaves like a `black hole'---photons go in but do not come out.
Part (ii) shows that a condensate never stops growing once it starts. 
Part (iii) states that a 
condensate must develop in finite time for any initial data $n\init$
with more photons than the maximal equilibrium $n_0=x$.  
Part (iv) means that for initial 
data bounded above by $n_0$, a condensate does not form.

According to Theorems 2.1 and 2.2, however, we see that 
the notion of a condensate is not strictly required
for mathematically discussing existence and uniqueness.
 The solution is determined by
the conditions imposed for $x\in(0,1]$, and it so happens that 
the total photon number can decrease due to an outward flux at $x=0^+$.

\begin{thm}
  [Large-time convergence] 
  \label{t:lim}
  Let $\nin$ satisfy \qref{init-data} for some $p\geq0$. 
Let $n$ be the unique global weak solution to (\ref{ne-}) associated 
with $\nin$.  Then there exists $\mu \geq 0$ such that
$$
  \lim_{t \to \infty} \| n(\cdot,t) - \nmu \|_1 = 0 \,, \qquad
  \text{where} \quad \nmu(x) = \frac{x^2}{x + \mu} \,.
$$
\end{thm}

The equilibrium $\nmu$ to which a solution converges  
  depends not only on $N[\nin]$, but also on details of $\nin$.
In some special cases, $\mu$ can be explicitly determined.

\begin{cor} \label{c:mulim}
Let $n$ be the global solution to (\ref{ne-}),
subject to initial data satisfying  $\nin(x) \geq x $ for $ x\in (0, 1]$. 
Then
$$
  \lim_{t \to \infty} n(x, t) = n_0(x)=x.
$$
Moreover,
\begin{equation}
  \label{nd:f}
  |n(x,t) - x| \leq \frac{1}{t} + \frac{2}{\sqrt{t}} \,, \qquad
  \text{for every $t>0$} \,.
\end{equation}
If  $\nin(x) \leq x $ for $ x\in (0,1]$, then
$$
  \lim_{t \to \infty} n(x,t) = n_\mu(x) = \frac{x^2}{x + \mu} \,,
$$
with $\mu$ uniquely determined by the relation 
\begin{equation}\label{mu}
N[\nin] = N[ n_\mu] =\frac12 -\mu +\mu^2\log\left(1+\frac1\mu\right)\ .
\end{equation}
\end{cor}

\begin{rem}
For the model  equation (\ref{eq:0}), these results provide a definite answer  
to the main issues of concern.  The main assertions are expected to hold true  
for the full Kompaneets equation (\ref{eq:ko}), and may be partially true for some 
extensions of the Kompaneets equation \cite{Rg03,Co71}.
Theorems 2.1 and 2.2 improve upon  Theorems 1 and 2 of \cite[p.~3839]{EHV98}
for the Kompaneets equation, in the sense that we impose no growth 
condition near $x=0$.   
Though for a model equation, Theorems 2.4 and 2.5 provide a 
theoretical justification of observations made previously, including the 
detailed singularity analysis given in \cite{EHV98}, the self-similar blow-up 
of the Kompaneets equation's solution in finite time \cite{JPR06}, as well as 
the classical result of  Levich and Zel'dovich \cite{LZ69} on shock waves in 
photon spectra.
\end{rem}

\begin{rem}
We remark that  the quantum entropy defined by 
$$
  H[n] = \int_0^1 [x n - x^2 \log(n)] \, dx
$$
satisfies 
$$
  H[n(t)] 
  + \int_0^t \int_0^1 
        n^2 \left( 1 - \partial_x \left( \frac{x^2}{n} \right) \right)^2 dx 
  \leq H[\nin] \,, \; \forall t>0,
$$
provided $H[\nin]<\infty.$  
As we have no need for this entropy dissipation inequality
in this paper, we omit the proof. 
We mention, however, that the entropy $H[n]$ is not sensitive to the presence
of the Bose-Einstein condensate.
\end{rem}

{\ifdraft
\vfil\pagebreak 
\fi}

\section{Uniqueness of weak solutions}\label{s:unique}

This section is primarily devoted to the proof of Theorem \ref{t.uniq}. 
At the end of this section we include an additional result, a strict $L^1$ contraction property,
which will be used in section~\ref{s:largetime}.

\subsection{Proof of Theorem~\ref{t.uniq}.}
Let $w=n-\bar n$ and $\hat n=n+\bar n-2x$.
Then from (\ref{we-}) for both $n$ and $\bar n$ it
follows that
\begin{equation}\label{w0+}
\int_{\Qplus}( w\, \Dt\psi -
(x^2\Dx w +\hat n w)
 \, \Dx\psi )\dX =0.
\end{equation}
The estimate \qref{wf} can be derived formally by using
a test function of form $\psi=x^p\one_{[0, t]}H(w)$,
where $H(w)$ is the usual Heaviside function
and $\one_E$ is the characteristic function of a set $E$.
This is not an admissible test function, however, and
instead we need several approximation steps.

For use below, we fix a smooth, nondecreasing cutoff function
$\chi:\R\to [0,1]$
with the property $\chi(x)=0$ for $x\le1$, $\chi(x)=1$ for $x\ge 2$,
and set $\chi_\eps(x)=\chi(x/\eps)$ for $\eps>0$. 
For any interval $I\subset[0,\infty)$ we define the space-time domains
\begin{equation}\label{d.QI}
Q_I = (0,1]\times I, \qquad\mbox{so}\ \ Q=Q_{(0,\infty)}=(0,1]\times(0,\infty).
\end{equation}

\smallskip
\noindent 1. (Steklov average in $t$.)
For $h\ne0$, the Steklov average $u_h$ of a continuous function
$u$ on $\Qplus$ is defined by
extending $u(x,t)$ to be zero for $t<0$, and setting
$$
u_h(x,t)=\frac{1}{h}\int_t^{t+h}u(x, s)\,ds,  \qquad (x,t)\in \Qplus.
$$
By density arguments,
the Steklov average extends to an operator with
the following properties:
First, for $1\le p<\infty$, if $u\in L^p\loc(\Qplus)$ then
$u_h\in L^p\loc(\Qplus)$ with weak derivative
\[
\D_t u_h = \frac{ u(\cdot,\cdot+h)-u(\cdot,\cdot)}h
\in L^p\loc(\Qplus).
\]
Moreover, one has $u_h\to u$ in $L^p\loc(\Qplus)$ as $h\to 0$.

Since 
\[
n,\bar n \in \Bplus:=\{n\mid n,\  \D_x n \in L^2\loc(\Qplus) \mbox{ with } n\ge0\},
\]
 it follows $\D_x^j\D_t^kw_h\in L^2\loc(\Qplus)$
for $j,k=0,1$, whence $w_h$ is continuous on $\Qplus$.
If $\psi$ is a $C^1$ test function with compact support
in $\Qplus$, the same is true for $\psi_{-h}$ if
$|h|$ is sufficiently small, and
a simple calculation with integration by parts and justified
by density of smooth functions shows that
$$
\int_{\Qplus} w\, \Dt(\psi_{-h})\dX =
\int_{\Qplus} w (\Dt\psi)_{-h}\dX
=\int_{\Qplus} w_h\Dt\psi\dX=-\int_{\Qplus} (\Dt w_{h}) \psi\dX.
$$
Substitution of this into \qref{w0+} and
treating the other term similarly, one finds
\begin{equation}\label{qq}
\int_{\Qplus} \Bigl((\Dt w_{h}) \psi +
(x^2\Dx w +\hat n w)_h
\, \partial_x \psi \Bigr)\dX=0.
\end{equation}

Recall $n, \bar n\in \Bplus$, hence
$\hat n w$ and $\D_x(\hat n w)$ are in
$L^1\loc(\Qplus)$,
whence $(\hat n w)_h$ is continuous in $\Qplus$.
By replacing $\psi(x,t)$ by
$\psi(x,t)\chi_\eps(t-\sigma)\chi_\eps(\tau-t)$ and taking $\eps\to0$
using dominated convergence, we find that
for any $C^1$ function $\psi$ with
compact support in $Q_{[0,\infty)}$,
\begin{equation}\label{qq+}
\int_{Q_{[\sigma,\tau]}} \Bigl( (\Dt w_{h}) \psi +
(x^2\Dx w +\hat n w)_h
\, \Dx \psi \Bigr)\dX=0,
\qquad\mbox{whenever $[\sigma,\tau]\subset(0,\infty)$.}
\end{equation}
By approximation, \qref{qq+} holds
for any $\psi \in W^{1,2}(\Qplus)$ supported in $Q_{[0,\infty)}$.
\\

\noindent  2. (Integrate in $t$.)
Define $\zeta(a)=  \int_0^a \chi_\epb(u)\,du$ as
a smooth, convex
approximation to the function $a\mapsto a_+$.
Since $\zeta$ is Lipshitz with $\zeta'(a)=1$ for $a>2\epb$, the composition
$\zeta(w_h)\in W^{1,2}\loc(\Qplus)$, with the weak derivatives
\[
\Dt\zeta(w_h)=\zeta'(w_h)\Dt w_h,\quad
\Dx\zeta(w_h)=\zeta'(w_h)\Dx w_h.
\]
We may now set $\psi(x,t)=\eta(x)(\zeta'\circ w_h)(x,t)$ in \qref{qq+}, where
$\eta$ is any $C^2$ function with compact support in $(0,1]$.
In what follows, we assume also that $\eta\ge0$ and $\eta'\ge0$ on $(0,1]$.
The function $t\mapsto \int_0^1 \eta(x)\zeta(w_h(x,t))\,dx$
is absolutely continuous for $t>0$, with
\begin{equation}\label{i.ez}
\Bigl.\int_0^1 \eta(x)\zeta(w_h(x,t))\,dx\Bigr|_{t=\sigma}^{t=\tau}
= \int_{Q_{[\sigma,\tau]}} \!\!\! \eta \zeta'(w_h)\Dt w_h 
= -\int_{Q_{[\sigma,\tau]}}
\!\!\! (x^2\Dx w +\hat n w)_h
\,\Dx(\eta \zeta'(w_h)) 
\end{equation}
whenever $[\sigma,\tau]\subset(0,\infty)$.
\\

\noindent 3. (Take $h\to0$.)
As $h\to0$, the hypotheses for weak solutions imply that
$n$, $\bar n\in L^1\loc(Q_{[0,\infty)})$. By consequence,
in $L^1\loc([0,\infty))$ we have
\begin{equation}
\int_0^1 \eta(x)\zeta(w_h(x,\cdot))\,dx \to
\int_0^1 \eta(x)\zeta(w(x,\cdot))\,dx .
\end{equation}
In fact, we will show that the right-hand side here is absolutely
continuous for $t>0$, from
studying the terms on the right-hand side of \qref{i.ez}:
First, as $h\to0$, in $L^2\loc(\Qplus)$
we have
\begin{equation}
x^2\Dx w_h\to x^2\Dx w.
\end{equation}
And along a subsequence $h_j\to0$,
$\zeta'(w_h)\to \zeta'(w)$ and $\zeta''(w_h)\to \zeta''(w)$
boundedly a.e.\ on compact subsets of $\Qplus$.
Hence in $L^2\loc(\Qplus)$ we have
\[
\Dx(\eta\,
\zeta'(w_h))=
\eta' \zeta'(w_h) + \eta\,\zeta''(w_h)\Dx w_h
\quad\to\quad
\eta' \zeta'(w) + \eta\,\zeta''(w)\Dx w .
\]
Since $\zeta'(w)\Dx w=\Dx\zeta(w)$, we find
\begin{equation}\label{h.t1}
 \int_{Q_{[\sigma,\tau]}} (x^2\Dx w_h)\Dx(\eta\,\zeta'(w_h))\dX
\to
 \int_{Q_{[\sigma,\tau]}} \Bigl(
x^2\eta'\Dx\zeta(w)+ x^2\eta\,\zeta''(w)(\Dx w)^2
\Bigr)\dX.
\end{equation}

Next we deal with the nonlinear term in \qref{i.ez}. Observe
\begin{equation}
\label{h.t2}
\int_{Q_{[\sigma,\tau]}} (\hat n w)_h \Dx(\eta\,\zeta'(w_h))\dX
=
 \int_0^\tau
\eta\,\zeta'(w_h) (\hat n w)_h (1,t) \,dt
-\int_{Q_{[\sigma,\tau]}}
\eta\,\zeta'(w_h) (\Dx(\hat nw))_h
\dX.
\end{equation}
Since $n,\bar n\in \Bplus$ we have $\hat nw$, $\Dx(\hat nw)\in L^1\loc(\Qplus)$,
and it follows that $(\hat nw)_h(1,\cdot)\to \hat nw(1,\cdot)$ in
$L^1\loc((0,\infty))$.
Moreover, $w(1,\cdot)\in L^2\loc((0,\infty))$ and
$\zeta'(w_h(1,\cdot))\to\zeta'(w(1,\cdot))$
boundedly a.e. on compact subsets of $(0,\infty)$
along a sub-subsequence of $h\to0$.
Thus we may pass to the limit on the right-hand side of \qref{h.t2},
integrate back by parts, and infer that the limit is
\begin{equation}
\label{h.t2b}
 \int_\sigma^\tau
\eta\,\zeta'(w) (\hat n w) (1,t) \,dt
-\int_{Q_{[\sigma,\tau]}}
\eta\,\zeta'(w) \Dx(\hat nw)
\dX
=
\int_{Q_{[\sigma,\tau]}} (\hat n w) (\eta'\zeta'(w)+\eta\,\zeta''(w)\Dx w)\dX.
\end{equation}
One can justify this equality using an additional argument
approximating $w=n-\bar n$ via smooth functions in $\Bplus$:
Note the intermeditate term
$\int_{Q_{[\sigma,\tau]}} (\hat n w) \Dx(\eta\,\zeta'(w))\dX$
does not make sense with only the regularity assumed for $w$,
due to insufficient integrability in time ($L^1$ for one factor, $L^2$
for the other). The right-hand side of
\qref{h.t2b} does make sense, however, since $\zeta'(w)$
and $w\zeta''(w)$ are bounded on the support of the integrand.

In sum, we find that in $L^1\loc((0,\infty))$ and for
a.e.~$\tau>\sigma>0$,
\begin{eqnarray}\label{i.ezlim}
 \Bigl.\int_0^1 \eta(x)\zeta(w(x,t))\,dx\Bigr|_{t=\sigma}^{t=\tau}
&=& - \int_{Q_{[\sigma,\tau]}} \Bigl(
x^2\eta'\Dx\zeta(w)+ x^2\eta\,\zeta''(w)(\Dx w)^2
\Bigr)\dX
\\
&& \nonumber \quad
-\ \int_{Q_{[\sigma,\tau]}} (\hat n w)
(\eta'\zeta'(w)+\eta\,\zeta''(w)\Dx w)\dX.
\end{eqnarray}

\noindent 4. (Take $\epb\to0$.)
Note that since $\zeta(w)$, $\Dx\zeta(w)\in L^2\loc(\Qplus)$,
we have $\zeta(w(1,\cdot))\in L^2\loc((0,\infty))$ and
\[
-\int_{Q_{[\sigma,\tau]}} \!\!\!
x^2\eta'\Dx\zeta(w)\dX =
-\int_\sigma^\tau \eta'(1)\zeta(w(1,t))\,dt
+ \int_{Q_{[\sigma,\tau]}} \!\!\! \zeta(w)\Dx(x^2\eta')\dX
\le
 \int_{Q_{[\sigma,\tau]}} \!\!\! \zeta(w)\Dx(x^2\eta')\dX.
\]
Since $\eta,\eta'\ge0$, $\hat n\ge -2x$, and $w\zeta'(w)\ge0$,
therefore
\begin{equation}
\label{h.zero}
\Bigl.\int_0^1 \eta(x)\zeta(w(x,t))\,dx\Bigr|_{t=\sigma}^{t=\tau}
\leq \int_{Q_{[\sigma,\tau]}} \Bigl(
\zeta(w) \Dx(x^2 \eta')
+2x \eta' w \zeta'(w)
- \hat n w \eta \,\zeta''(w)(\Dx w)
\Bigr) \dX.
\end{equation}

Now we take the limit $\epb\downarrow0$,
for which we have
$\zeta\circ w\uparrow w_+$ and $w(\zeta'\circ w)\uparrow w_+$
pointwise.  Moreover, $w\zeta''\circ w = (w/\epb)\chi'(w/\epb)$
is bounded and converges to zero a.e.
Since $\eta \hat n\Dx w\in L^1(\Qplus)$, by dominated convergence
the last term in \qref{h.zero} tends to zero, and we derive
\begin{equation}\label{e:new1a}
\int_0^1 \eta\, w_+(x,\tau)\,dx \le
\int_0^1 \eta\, w_+(x,\sigma)\,dx +
\int_{Q_{[\sigma,\tau]}} (x^2 \eta''+4x\eta')w_+ \dX,
\end{equation}
for a.e.\ $\tau>\sigma>0$. Due to assumption \qref{con-} on weak
solutions, now we can take $\sigma\to0$ and conclude that this inequality holds
also with $\sigma=0$.
\\

\noindent 5. (Make Gronwall estimate.) Finally, we take $\eta$ of
the form $\eta(x)=x^p\chi_\eps(x)$,
where $p\ge0$ is the exponent for which we assume $x^pn\init\in L^1(0,1)$.
Observe that
\begin{equation}\label{d.etap}
x\eta'= x^p(p\chi+(x/\eps)\chi'),
\quad
x^2\eta'' = x^p (p(p-1)\chi+2p (x/\eps)\chi' + (x/\eps)^2 \chi''),
\end{equation}
where the arguments of $\chi$, $\chi'$ and $\chi''$ are $x/\eps$.
As $\eps\to0$, since we assume $x^p n$ and $x^p \bar n$ are in
$L^1(Q_{[0,T]})$ for any $T>0$,
we infer by monotone and dominated convergence that
\begin{equation}\label{e:new2a}
\int_0^1 x^p w_+(x, \tau) \,dx  \leq \int_0^1 x^p w\init_+(x)\,dx
+c_p\int_0^\tau \int_0^1 x^p w_+(x,t)\,dx\,dt,
\end{equation}
for a.e. $\tau>0$,  with $c_p=p(p-1)+4p=p(p+3)$.
{Denoting the (absolutely continuous)
right hand side of \qref{e:new2a} by $U(\tau)$, we have
$$
U(\tau)=U(0) +\int_0^\tau U'(s)\, ds \leq
U(0) +c_p \int_0^\tau U(s)ds,
$$
and Gronwall's inequality implies that
$$
U(\tau) \leq e^{c_p \tau} \int_0^1 x^p\,w\init_+(x)\,dx,
$$
for all $\tau>0$.  This proves \qref{wf}.
Clearly, $\nin\le\bnin$ implies $n\le\bar n$,
by virtue of \qref{wf}.
}

\subsection{Strict $L^1$-contraction}
From Corollary~\ref{c:uniq} it easily follows 
that  weak solutions of \qref{ne-} enjoy
the $L^1$ contraction property mentioned in \qref{e:l1c}.
For use in section~\ref{s:largetime} below, we strengthen this
to the following {\em strict} $L^1$ contraction property for 
$C^1$ solutions that {\em cross transversely}.

\begin{lem}\label{contraction}
Let $n, \bar n$ be nonnegative solutions to
(\ref{ne-})  with respect to initial data $\nin, \bnin$ that are in
$L^1(0, 1)\cap L^\infty(0, 1)$,  then for a.e.\ $t>0$,
\begin{equation}\label{con}
\|n(\cdot,t)-\bar n\|_{L^1(0, 1)} \leq \| \nin- \bnin \|_{L^1(0, 1)}.
\end{equation}
Moreover, assuming
the solutions 
$n$ and $\bar n$ are $C^1$ in $(0,1)\times [0,\infty)$,
and that for some $t_0\ge0$, $n(\cdot,t_0)$ and $\bar n(\cdot,t_0)$ cross transversely at least once on $(0, 1)$, then
for all $t>t_0$ we have
\begin{equation}\label{con+}
\|n(\cdot,t)-\bar n(\cdot,t)\|_{L^1(0, 1)} < \| n(\cdot,t_0)- \bar n(\cdot,t_0) \|_{L^1(0, 1)}.
\end{equation}
\end{lem}

\begin{proof} 
{The $L^1$-contraction estimate \qref{con} follows directly from Corollary~\ref{c:uniq} with $p=0$.
In order to prove (\ref{con+}), it suffices to treat the case $t_0=0$ for $t>0$ sufficiently
small, and assume the right-hand side is finite.  
 Let $w=n -\bar n $ and $\hat n = n+\bar n - 2x$. 
If  $n$ crosses $\bar n$ transversely at $(x_0, 0)$, then the regularity of the solution implies that there exists a nondegenerate rectangle $\Sigma_0= [x_0-\hat \delta, x_0+\hat \delta]\times [0, \delta]$ 
such that $w(x_0, 0)=0$ and $\Dx w \not=0$  in $\Sigma_0$.  We suppose $\Dx w(x_0, 0)>0$
(relabeling $n$ and $\bar n$ if necessary),  whence $ \Dx w \geq c_1>0$ in $\Sigma_0$, so
$w(x_0+\hat \delta, t)>c_1\hat \delta>0$ and $w(x_0-\hat \delta, t)<- c_1\hat \delta<0.$

We follow the proof of Theorem~\ref{t.uniq} up to (\ref{h.zero}), finding that
for $0<\sigma<\tau<\delta$,
\begin{align}
\label{h.zero+}
& \Bigl.\int_0^1 \eta(x)\zeta(w(x,t))\,dx\Bigr|_{t=\sigma}^{t=\tau}    
\\\notag
&\qquad  \leq \int_{Q_{[\sigma,\tau]}} \Bigl(-x^2\eta \zeta''(w)(\partial_x w)^2 
 + \zeta(w) \Dx(x^2 \eta')
+2x \eta' w \zeta'(w)
- \hat n w \eta \,\zeta''(w)(\Dx w)
\Bigr) \dX.
\end{align}
Here we include a term $-x^2\eta \zeta''(w)(\partial_x w)^2$ from \qref{i.ezlim} that
was dropped in \qref{h.zero}.  
This identity is valid for any $C^2$ function $\eta\geq 0$ with 
compact support in $(0, 1]$ and with $\eta'\geq 0$. 
We may require $x^2\eta\geq c_2>0$
in $\Sigma_0$. Therefore, taking $\epb\downarrow0$, we find
\begin{align*}
 \int_{Q_{[\sigma,\tau]}} x^2\eta \zeta''(w)(\partial_x w)^2 &\geq 
\int_{\Sigma_0\cap Q_{[\sigma,\tau]} }  c_1c_2\zeta''(w)\partial_x w 
\\&
= \int_{\sigma}^{\tau} c_1c_2\zeta'(w)\Big|^{x_0+\hat \delta}_{x_0-\hat \delta}\, dt 
\to c_1c_2(\tau-\sigma)>0\,.
\end{align*}
When  taking  the limit $\epb\downarrow0$,we also have  
$\zeta\circ w\uparrow w_+$ and $w(\zeta'\circ w)\uparrow w_+$
pointwise.  Moreover, $w\zeta''\circ w = (w/\epb)\chi'(w/\epb)$
is bounded and converges to zero a.e.
Since $\eta \hat n\Dx w\in L^1(\Qplus)$, by dominated convergence
the last term in \qref{h.zero+} tends to zero, and we derive
\begin{equation}\label{e:new1aa}
\int_0^1 \eta\, w_+(x,\tau)\,dx \le
\int_0^1 \eta\, w_+(x,\sigma)\,dx +
\int_{Q_{[\sigma,\tau]}} (x^2 \eta''+4x\eta')w_+ \dX-c_1c_2(\tau-\sigma)\,.
\end{equation}
Finally, we take $\eta$ of the form $\eta(x)= \chi_\theta(x)$ with $\theta<x_0-\hat \delta$.  Observe that
\begin{equation}\label{d.etap+}
x\eta'= (x/\theta)\chi',
\quad
x^2\eta'' = (x/\theta)^2 \chi'',
\end{equation}
where the arguments of $\chi$, $\chi'$ and $\chi''$ are $x/\theta$.
As $\theta \to0$, since we assume $ n$ and $ \bar n$ are in
$L^1(Q)$, we infer by monotone and dominated convergence that
\begin{equation}\label{e:new2a+}
\int_0^1  w_+(x, \tau) \,dx  \leq \int_0^1  w_+(x,\sigma)\,dx-c_1c_2(\tau-\sigma) <
\int_0^1  w\init_+(x)\,dx\,,
\end{equation}
where the last inequality follows by applying Theorem~\ref{t.uniq} with $t=\sigma$
and $p=0$.  Adding this result together with \qref{wf} with $p=0$ and
  $n$  interchanged with $\bar n$, we obtain  (\ref{con+}).
}
\end{proof}

{\ifdraft
\vfil\pagebreak 
\fi}

\section{Existence of weak solutions}\label{s:exist}

The existence result in Theorem~\ref{t.exist}
is proved through three main approximation steps:
\begin{itemize}
\item[(i)] Approximate the rough initial data $\nin\in L^1(x^pdx)$ 
by smooth data $\nin_\kappa$ that is strictly positive and bounded. 
\item[(ii)] Truncate the problem \qref{ne-}  
to $x\in[\eps,1]$ with $\eps>0$, resulting in a strictly parabolic problem
at the cost of needing to impose an additional boundary condition at $x=\eps$. 
\item[(iii)] Further approximate by 
cutting off the nonlinearity in the flux near the boundary $x=1$,
resulting in a problem with linear boundary conditions. 
\end{itemize}
Passing to the limit in the various approximations involves
compactness arguments and uniform estimates that are based on
energy estimates and Gronwall inequalities.  Step (iii) is comparatively 
straightforward and its analysis is relegated to Appendix B. 
We deal with steps (i) and (ii) in the remainder of this section.

\subsection{Smoothing the initial data.}
Consider fixed initial data $n \init$ in $L^1(x^p dx)$.
We regularize the given initial data to obtain a family of  functions $n_\kappa\init$
for small $\kappa>0$, which are smooth on $[0,1]$ and  positive
on $(0,1]$, with the following properties:
\begin{align}
 \int_0^1 x^p|n_\kappa\init - \nin|\,dx\to0 &\quad\mbox{as $\kappa\to0$}\,,
\label{ap}\\
 \nin_\kappa(x) = \kappa x^2\ , &\qquad{0<x<\kappa}\,, 
\label{ap0}\\
 \nin_\kappa(x) = \frac{\kappa x^2}{\kappa x + 1}\,, &\qquad {1-\kappa<x<1}\,.
\label{ap1}
\end{align}
(The properties in \qref{ap0} and \qref{ap1} are conveniences so that we get 
compatible initial data in the approximation steps to follow below.)
The desired regularization can be achieved through mollification: 
Let  $\rho$ be a smooth, nonnegative function on $\R$ with
support contained in $(-1, 1)$ and total mass one.  Define
\begin{equation}\label{d:rhochi}
\rho_\kappa(x)= \kappa\inv \rho(x/\kappa), \qquad 
\chi(x)= \int_{-\infty}^x \rho(z)\,dz
\end{equation}
(note $\chi(x)=0$ for $x<-1$ and $\chi(x)=1$ for $x>1$), and require that
\[
x^p n_\kappa\init(x) =
\int_{2\kappa}^{1-2\kappa}  \rho_\kappa(x-y) \,y^p n\init(y)\,dy
+ \frac{\kappa x^{2+p}}{1+\kappa x\,\chi(4x-2)} \ .
\]
The integral
term vanishes when $x<\kappa$ or $x>1-\kappa$, and there is no singularity near $x=0$.

For this regularized initial data, our goal is to prove the following result. 
\begin{prop} \label{t:existn}
For every small enough $\kappa>0$, there exists a  weak solution $n_\kappa$ of \qref{ne-}
with initial data $\nin=\nin_\kappa$, having $n_\kappa\in C([0,\infty),L^1((0,1]))$.
\end{prop}

\subsection{Truncation}
To obtain $n_\kappa$, we regularize by truncating the domain away from the
origin, thus removing the degenerate parabolic nature of the problem.   
In other words, we will study classical solutions of the following problem
for small $\eps>0$:  
In terms of the (left-oriented) flux
\begin{equation}\label{d.Jep}
J_\eps=x^2\Dx\nep-2x\nep+\nep^2\,,
\end{equation}
we seek a solution to the problem
 \begin{subequations}\label{ne}
\begin{align}
 \Dt \nep  &= \Dx J_\eps \,, 
&\qquad x\in (\epsilon, 1), \ t\in(0,\infty)\,,
\label{e:nep1}
\\
 \nep &=n\init_\kappa\,, &\qquad x\in (\epsilon, 1)\,, \ t=0\,,
\label{e:nep2}
\\
0&=J_\eps\,, 
&x=1\,,\ t\in[0,\infty)\,,
\label{e:nep3}
\\
0&= \eps^2\Dx\nep-2\eps\nep\,,
&x=\eps\,,\ t\in[0,\infty)\,.
\label{e:nep4}
\end{align}
\end{subequations}
The boundary condition \qref{e:nep4} says $J_\eps=\nep^2$ at $x=\eps$.
As will be seen in section~\ref{s:bec} below, this boundary condition is well-adapted
to proving the conservation identity for photon number in Theorem~\ref{t:cond}.  
An important point to note, however, is that
the uniqueness result of Theorem~\ref{t.uniq} 
shows that the solution of \qref{ne-} {\em does not depend}
on the choice of this boundary condition in \qref{e:nep4}.

For fixed small $\eps>0$, the following global existence result for
classical solutions of \qref{ne} is proved in Appendix B.
Note that due to \qref{ap0} and \qref{ap1}, 
the boundary conditions \qref{e:nep3}--\qref{e:nep4}
hold at $t=0$ whenever $0<\eps<\kappa$. 
\begin{prop}\label{t:regexist} 
Let $\nin_\kappa$ be smooth and positive on $(0,1]$ and satisfy
\qref{ap0}--\qref{ap1}.
Then for any sufficiently small $\eps>0$, 
there is a global classical solution $\nep$ of \qref{ne}, 
smooth in the domain
\begin{equation}\label{d:Qeps}
Q^\eps:=(\eps,1)\times(0,\infty)\,, 
\end{equation}
with $\nep$, $J_\eps$ and $\Dx\nep$ globally bounded and continuous on 
$\bar Q^\eps=[\eps,1]\times[0,\infty)$.
\end{prop}
From this result, we will derive Proposition~\ref{t:existn} by taking $\eps\downarrow0$
after establishing a number of uniform bounds on the solution
$\nep$ of \qref{ne}.
The global bounds stated in Theorem~\ref{t.exist} will
follow directly from corresponding uniform bounds on $\nep$,
which are proved in Lemmas~\ref{l:uss} and \ref{l:ec}
and are inherited by $n_\kappa$.

\subsection{Uniform estimates for the truncation.}
The first few uniform estimates that we establish
on the solution $\nep$ of \qref{ne} are pointwise estimates
that arise from comparison principles.

\begin{lem} \label{0m}
We have $\nep(x, t)>0$ for every $(x, t)\in[\epsilon, 1]\times [0, \infty)$\,.
\end{lem}
\begin{proof}  Recall $\min_{[\eps,1]} \nin_\kappa>0$. 
If we suppose the claim fails, then $0<t^*<\infty$, where
$$
t^*=\sup \{t \mid \nep(x,t)>0 \text{ for all $x\in [\epsilon, 1]$}\}\,.
$$
By continuity, there exists $X^*=(x^*, t^*)$ with $x^*\in [\eps, 1]$ such that
$\nep(X^*)=0$.  We claim first that $x^*\ne\eps$ or $1$.
If $x^*=\epsilon$ or $1$, by the
strong maximum principle \cite{PW84}, we must have $0\ne\Dx \nep(X^*)$,
but this violates the boundary conditions \qref{e:nep3}--\qref{e:nep4}.

Thus $\eps<x^*<1$, but this is also not possible due to rather standard
comparison arguments:  There exists $\delta>0$ such that
$\delta$ is less than the minimum of $\nep(\eps,t)$, $\nep(1,t)$
and $\nep(x,0)$ whenever $0\le t\le t^*$ and $x\in[\eps,1]$.
Setting $w=e^{3t}\nep$, we find that $w>\delta$ at $t=0$,
and there is some first time $\hat t\in(0,t^*)$
when $w(\hat X)=\delta$ for some $\hat X=(\hat x,\hat t)$
with $\hat x\in(\eps,1)$.
Then $\Dt w\le0$, $\Dx w=0$ and $\Dxx w\ge0$ at $\hat X$,
but computation then shows $\Dt w \ge w=\delta>0$.
This finishes the proof.
\end{proof}

Next we establish a {universal upper bound} on our solution
of \qref{ne}. We do this by establishing that the function defined by
\begin{equation}\label{d:S}
S(x,t) =x +\frac{1-x}{t} +\frac{2}{\sqrt t}
\end{equation}
is a universal super-solution.
This fact depends essentially on the hyperbolic nature of our
problem at large amplitude---Note that the middle term $(1-x)/t$
is a centered rarefaction wave solution of the equation $\Dt n-2n\Dx n=0$.

\begin{lem} \label{l:uss} 
We have
\[
 \nep(x, t) < S(x, t) \quad\mbox{ for all $(x,t)\in \bar Q^\eps$}\,.
\]
Furthemore, there exists $\tau_1>0$, depending only on $\sup \nin_\kappa$, such that
\[
 \nep(x, t) < S(x, t+\tau_1) \quad\mbox{for all $(x,t)\in \bar Q^\eps$}\,.
\]
\end{lem}

\begin{proof} Let us write
$ L[n] :=\Dt n- x^2\Dxx n-2n(\Dx n-1).  $
 Then a simple calculation gives
 $$
 L[S]=\frac{1-x}{t^2} +\frac{2x}{t} +3t^{-3/2}>0\,.
 $$
 Hence with $v=S-\nep$, we have
\begin{equation}\label{i:LLnS}
 L[S]- L[\nep]= \Dt v -x^2 \Dxx v -\Dx((\nep+S)v)+2v>0\,.
\end{equation}
By continuity we have $\min_x v(x,t)>0$ for small $t>0$, and we claim that
this continues to hold for all $t>0$. If not, there is a first time
$\hat t$ when it fails, and some $\hat X=(\hat x,\hat t)$ with
$\hat x\in[\eps,1]$ where $v(\hat X)=0$.
By \qref{i:LLnS} it is impossible that $\hat x\in(\eps,1)$.  If $\hat x=\eps$,
then $v=0$ and $\Dx v\ge0$ at $\hat X$.  But due to the boundary condition
\qref{e:nep4} we find that at
$(x,t)=(\eps,\hat t)$,
\[
 0\le \eps \Dx v  =\eps \Dx S -2S
 = \eps\left(1-\frac1t\right)
-2\left(\eps + \frac{1-\eps}t+\frac2{\sqrt t}\right) <0\,.
\]
So $\hat x\ne\eps$.
 On the other hand, if $\hat x=1$, we would have
$v=0$ and $\Dx v\le0$ at $\hat X$. But then,
at $(x,t)=(1,\hat t)$
we find by \qref{e:nep3} that
\[
0\ge \Dx v  = \Dx S + S^2 -2S
=
1-\frac1t
+\left(1+\frac2{\sqrt t}\right)^2
-2\left(1+\frac2{\sqrt t}\right)
= \frac 3t > 0\,.
\]
Thus $\hat x\ne1$, and the result $S-\nep>0$ follows.  
Furthermore, if $\sup \nin_\kappa \le 2/\sqrt{\tau_1}$ then
$$
\min_x (S(x, \tau_1) -\nin_\kappa(x)) >\frac{2}{\sqrt{\tau_1}} 
-\|n\init_\kappa\|_{L^\infty} \ge0\,.
$$
Then the above procedure shows that $S(x, t+\tau_1)$ is also a super-solution. 
\end{proof}

The next result bounds $\Dx\nep$ from below. This is again a typical
kind of estimate for the hyperbolic equation $\Dt n-2n\Dx n=0$.
\begin{lem}\label{l:ec}(Oleinik-type inequality) 
We have $$
\Dx\nep(x,t) \geq -\frac{4}{t} \qquad\mbox{for all $(x,t)\in \bar Q^\eps$\,.} 
$$
Futhermore, there exists $\tau_2>0$, depending only on 
$\inf \Dx\nin_\kappa$, such that
\[
\Dx\nep(x,t) \geq -\frac{4}{t+\tau_2}
  \qquad\mbox{ for all $(x,t)\in \bar Q^\eps$\,.}
\]
\end{lem}
\begin{proof}
Let $w=\partial_x \nep$ with $\nep$ being the solution of (\ref{ne}).
Differentiation of (\ref{e:nep1}) shows that $w$ satisfies
 \begin{subequations}\label{w}
  \begin{align}
\Dt w & =x ^2 \Dxx w +2(\nep +x)\Dx w +2w(w-1)\,,
& (x,t)\in Q^\eps\,,\\
\eps^2 w(\epsilon, t) &=2 \eps\nep(\epsilon, t)\,, & t>0\,,
\\
 w(1, t)&=2\nep(1, t) -\nep(1, t)^2\,,
 & t>0\,.
\end{align}
\end{subequations}
We claim that $z=-4/t$ is a sub-solution of this problem.
Set $U=w-z=w+{4}/{t}$. A direct calculation gives
$$
\Dt U-x ^2 \Dxx U-2(\nep+x )\Dx U-2(w-4/t)U+2U=\frac{4}{t^2}(7+2t)>0\,.
$$
Note that $U(x, t)>0$ for $t>0$ small, because $\Dx \nep$ is continuous on $\bar Q^\eps$.
Then $U(x, t)>0$ for all $x$ and $t$ as long as it is so at $x=\epsilon$ and $x=1$.  At $x=\epsilon$, we have
$$
U(\epsilon, t)=w(\epsilon, t)+\frac{4}{t}=\frac{2}{\epsilon}\nep(\epsilon, t)+\frac{4}{t}>0\,.
$$
On the other hand, at $x=1$,  we have
\[
U(1, t) =w(1, t)+\frac{4}{t}
          =2\nep(1, t) -\nep(1, t)^2+\frac{4}{t}\,.
\]
If $0\le\nep(1,t)\le2$ then $U(1,t)\ge 4/t$, and otherwise,
$ 2< \nep(1, t)\leq S(1, t)=1 +2t^{-1/2}$,
hence
\[
U(1,t)\ge 2S(1,t)-S(1,t)^2+\frac4t=1\,.
\]
Therefore $U(1,t)>0$.  

  Provided $\tau_2>0$ is sufficiently small so that
$\Dx\nin_\kappa > -4/\tau_2$, we have
$$
\min_x\left\{ w(x, 0) +\frac{4}{\tau_2}\right\} >0\,,
$$
hence the above procedure shows that $w(x, t) \geq -{4}/({t+\tau_2})$ 
for all $(x, t)$ under consideration.
This concludes the proof.
\end{proof}

We now turn to obtain some compactness estimates that will be needed to
establish convergence  as $\eps\downarrow0$.
First we establish equicontinuity in the mean for solutions of (\ref{ne}).
\begin{lem}\label{l:bvlem} 
For each $t>0$, we have
\begin{equation}
\int_\epsilon^1|\Dx\nep|\,dx \leq K_1(t),
\qquad {K_1(t)=1+\frac2{\sqrt {t+\tau_1}}+\frac{8}{t+\tau_2}}\,,
\end{equation}
and for all $t>0$ and all small $h>0$,
\begin{equation}\label{e:halp}
\int_\epsilon^1|\nep(x, t+h)-\nep(x, t)|\,dx \leq K_2(t) h^{1/2}\,,
\end{equation}
where $K_2(t)$ is a decreasing function of $t$ {with $K_2(0)$ bounded} by a constant
depending only on $\sup \nin_\kappa$ and $\inf \Dx\nin_\kappa$.
\end{lem}
\begin{proof}  With $\tau_1$ and $\tau_2$ determined by the previous two lemmas, set
$$
u(x, t)=\nep(x, t) + \frac{4x}{t+\tau_2}\,,
$$
Then by Lemma~\ref{l:ec}, $u$ is a non-decreasing function of $x$, satisfying
$\Dx u>0$.
We have
\begin{align}\label{bv}
\int_\epsilon^1|\Dx\nep|\,dx & =\int_\epsilon^1
\left|\Dx u-\frac{4}{t+\tau_2}\right|dx
\leq \frac{4}{t+\tau_2} +\int_\epsilon^1 \Dx u \,dx \\ \notag
& \le \frac{8}{t+\tau_2}+ \nep(1, t) \le \frac{8}{t+\tau_2} +S(1,t+\tau_1) = K_1(t)\,.
\end{align}
This proves the first estimate of the lemma.

We next prove the bound \qref{e:halp}.
Fix any $t>0$ and consider $h>0$ small.  Suppressing the dependence
on $t$ and $h$, we set
$$
v(x)= \nep(x, t+h)-\nep(x, t)\,, \qquad x\in[\eps,1]\,,
$$
and observe that
\begin{equation}\label{i:vbd}
\|v\|_\infty \le 2\|S(\cdot,t+\tau_1)\|_\infty\,,
\qquad
\int_\eps^1|\Dx v|\,dx \le 2 K_1(t)\,.
\end{equation}
We proceed by approximating $|v(x)|$ by $\phi(x)v(x)$ where
$\phi$ is obtained by mollifying $\sgn v(x)$.
Let $\rho$ be a smooth, nonnegative function on $\R$ with
support contained in $(-1, 1)$ and total mass one,
and $\alpha>0$ be a parameter. (We will take $\alpha=\frac12$ below.)
We define $\rho_h(x)=h^{-\alpha}\rho(x/h^\alpha)$, and set
\begin{equation}\label{d:phi}
\phi(x)=\int_\eps^1 \rho_h (x-z) \sgn v(z) \,dz\,.
\end{equation}
To bound the integral of $|v(x)|$ over $[\eps,1]$,
we bound integrals over the sets
$$
I_h=[\eps+h^\ap, 1-h^\ap]\,, \qquad
 \hat I_h = [\eps,\eps+h^\ap]\cup[1-h^\ap,1]\,,
$$
writing
\begin{equation}\label{e:vdec}
\int_\eps^1|v(x)|\,dx =
\int_\eps^1 \phi(x)v(x)\,dx
+ \int_{I_h} \Bigl(|v(x)|-\phi(x)v(x)\Bigr)\,dx
+\int_{\hat I_h} \Bigl(|v(x)|-\phi(x)v(x)\Bigr)\,dx\,.
\end{equation}
Since $|\phi|\le1$,
the third term is bounded using the first estimate in \qref{i:vbd} as
\begin{equation}\label{i:term3}
\int_{\hat I_h} \bigl||v(x)|-\phi(x)v(x)\bigr|\,dx \le
8h^\ap \|S(\cdot,t)\|_\infty \,.
\end{equation}

We next estimate the middle term in \qref{e:vdec}.
For $x\in I_h$, we compute
\[
|v(x)|-\phi(x)v(x)
= \int_\R \rho_h(x-z) \Bigl(|v(x)|-v(x)\sgn v(z)\Bigr)\,dz\,.
\]
Noting that $||a|-a\sgn b|\le 2|a-b|$ for any real $a$, $b$, we have
\[
\bigl||v(x)|-v(x)\sgn v(z)\bigr| \le 2|v(x)-v(z)|\,.
\]
Integrating over $I_h$, we find
\begin{align*}
\int_{I_h}\bigl| |v(x)|-\phi(x)v(x) \bigr|  \,dx
&\le 2\int_{I_h}\int_{|y|<h^\ap} \rho_h(y) |v(x)-v(x-y)|\,dy\,dx
\\
&\le
2\int_{I_h}
\int_{|y|<h^\ap}
\rho_h(y)|y|
\int_0^1|\Dx v(x-ys)|\,ds
\,dy\,dx\,.
\end{align*}
Now we integrate first over $x$, note $x-ys\in[\eps,1]$ and use \qref{i:vbd},
and note that $|y|\le h^\ap$ and $\rho_h$ has unit integral.
We infer that
\begin{equation}
\label{i.vmid}
\int_{I_h}\bigl| |v(x)|-\phi(x)v(x) \bigr|  \,dx
\le 2h^\ap
\int_\eps^1|\Dx v|\,dx \le 4h^\ap K_1(t)\,.
\end{equation}

Finally, we bound the first term in \qref{e:vdec}.
Multiply equation (\ref{ne}a) by $\phi$ and integrate over
$(\epsilon, 1)\times (t, t+h)$. Integration by parts yields
\begin{align}
\label{aa}
\int_\epsilon^1 \phi v(x)\,dx
=\int_t^{t+h} \int_\epsilon^1
(\Dx\phi)
\left(
-x^2 \Dx \nep -
\nep^2+2 x \nep
\right)\,dx\,d\tau
-\int_t^{t+h} \phi(\eps)\nep^2(\epsilon, \tau) \,d\tau\,.
\end{align}
Note that $|\phi|\le1$ and $|\Dx\phi|\leq
h^{-\alpha}\|\rho'\|_1$
By virtue of $0\leq \nep\leq S$, we have
\begin{align*}
\int_\epsilon^1 \phi v\,dx & \leq
h^{-\alpha}\|\rho'\|_1
\int_t^{t+h}\int_\epsilon^1
( |\Dx\nep| + S^2+2S)
\,dx\,d\tau
+\int_t^{t+h}S(\epsilon, \tau)^2\,d\tau
\\
& \leq h^{1-\ap} \|\rho'\|_1(K_1(t)+3 \|S(\cdot,t+\tau_1)\|_\infty^2)
+ h \|S(\cdot,t+\tau_1)\|_\infty^2\,.
\end{align*}
Assembling all the bounds on the terms in \qref{e:vdec} above,
we obtain
\[
\int_\eps^1|v(x)|\,dx \le
\|S(\cdot,t+\tau_1)\|_\infty^2 (8h^\ap+h+3 h^{1-\ap}\|\rho'\|_1)
+K_1(t)( 4h^\ap+ h^{1-\ap}\|\rho'\|_1)\,.
\]
Choosing $\ap=\frac12$ and determining $K_2(t)$ to correspond,
the result in the lemma follows.
\end{proof}

Finally, we have the following energy estimate.
{
\begin{lem} \label{lem2.5} 
For any $t>s>0$,
\begin{equation}\label{en}
\int_\eps^1 \nep^2(x, t)\,dx +\int_s^t\int_\eps^1 [\nep^2 +x ^2(\Dx \nep)^2]\,dx\, d\tau
\leq \int_\eps^1 \nep^2(x, s)\,dx+ \frac{8}{3}(t-s)\,.
 \end{equation}
\end{lem}
\begin{proof} From equation (\ref{e:nep1}) and the boundary conditions
\qref{e:nep3}--\qref{e:nep4} it follows
$$
\frac{d}{dt}\int_\epsilon^1 \nep^2\,dx=
- 2\int_\epsilon^1 (\Dx\nep) J_\eps\, dx- 2\nep^3(\epsilon, t)=
-2\int_\eps^1[\nep^2+x ^2(\Dx\nep)^2]dx +\Gamma(t)\,,
$$
with 
\begin{align*}
\Gamma(t) & =-\frac{2}{3}\nep^3(1, t)-\frac{4}{3}\nep^2(\eps, t)
+2 \nep^2(1, t) -2\epsilon \nep^2(\eps, t)
\leq \max_{u>0} \left(-\frac23 u^3+ 2u^2\right) = \frac 83\,.
\end{align*}
Hence, the claimed estimate follows by integration in time.
\end{proof}
}

\subsection{Proof of  Proposition \ref{t:existn}}

We now show a solution of  (\ref{ne-}) does exist 
for initial data $\nin_\kappa$ as prepared in subsection 4.1.
Let $n_\epsilon$ be our solution of (\ref{ne}), for small $\eps>0$.

Recalling the uniform estimates $0<n_\epsilon\leq S(x, t+\tau_1)$
from Lemmas~\ref{0m} and \ref{l:uss},  and using
Lemma~\ref{l:bvlem}, we see
the family$\{n_\epsilon\}$ is uniformly bounded and equicontinuous in the mean
on any compact subset of $(0, 1]\times[0, \infty)$. 
Consequently, we may extract a sequence 
$\epsilon_k \downarrow 0$ as $k\to \infty$, 
such that  for each $a\in(0,1)$ and $T>0$, 
$n_{\epsilon_k}$ converges to some function $n$,
boundedly almost everywhere in $[a,1]\times[0,\infty)$
and in $C([0,T]; L^1([a,1]))$, with
\begin{equation} \label{e:hcon}
\int_a^1|n(x, t+h)-n(x, t)|\,dx \leq Ch^{1/2}\,.
\end{equation}
Actually, this $C$ is independent of $a$, so (\ref{e:hcon}) holds also with $a=0$.
Moreover, due to  (\ref{en}) we can ensure that
$$
x  \partial_x n_{\epsilon_k} \to x \partial_x n \qquad \mbox{weakly in }\ L^2\loc(Q)\,.
$$

We claim that $n$ is a weak solution of (\ref{ne-}). 
Multiply (\ref{ne}a) by a smooth test function $\psi$ with
compact support in $(0,1]\times(0,\infty)$,
and integrate over $(\epsilon, 1)\times (0,
\infty)$ with integration by parts to obtain, for small enough $\eps$,
\begin{equation}\label{we}
\int_0^\infty \int_\epsilon^1\Bigl(\nep  \Dt\psi -
(x^2 \Dx \nep+ \nep^2-2x \nep)\D_x\psi\Bigr)\dX=0\,.
\end{equation}
Setting $\epsilon=\epsilon_k$ and letting $k\to
\infty$, we conclude that 
\begin{equation}\label{w0}
\int_0^\infty \int_0^1\Bigl(n \Dt\psi -(x^2 \partial_x n+n^2-2x
n)\D_x\psi\Bigr)\dX=0
\end{equation}
for all smooth test functions $\psi$.  
By completion, we infer
that (\ref{w0}) holds for all $\psi$ merely in $H^{1}$
with compact support in $Q$.
Hence $n$ is a weak solution as claimed.

Since there may exist at most one such solution of (\ref{ne-}), we conclude that the whole family $\{n_\epsilon\}$ converges to $n$, as $\epsilon \to 0$. 
This ends the proof of Proposition~\ref{t:existn}.
\hfill\(\Box\)

\medskip
\subsection{Proof of Theorem~\ref{t.exist}.}
Our next task is to complete the proof of Theorem \ref{t.exist} 
by studying the solutions $n_\kappa$ from Proposition~\ref{t:existn}
in the limit $\kappa\to0$. 

By the Gronwall inequality from  (\ref{wf}),  for any fixed $T>0$, and small $\kappa_1,\kappa_2>0$,
\begin{equation}\label{ww+}
\sup_{t\in[0, T]} \int_0^1x^p |(n_{\kappa_1}-n_{\kappa_2})(x,t)|\,dx \leq
e^{c_pt} \int_0^1 x^p |(n_{\kappa_1}\init- n_{\kappa_2}\init)(x)|\,dx\,.
\end{equation}
This with (\ref{ap}) implies that $n_\kappa$ is a Cauchy sequence in $C([0, T]; L^1(x^pdx))$, and therefore there is a function $n \in C([0, T]; L^1(x^pdx))$ such that  
$ \lim_{\kappa \to 0} n_\kappa=n$.  
From the local-in-time energy estimate (\ref{en}) we have  
\begin{equation*}
\int_0^1 n_\kappa^2(x, t)\,dx +\int_s^t\int_0^1 [n_\kappa^2 +x ^2(\Dx n_\kappa)^2]\,dx\, d\tau
\leq \int_0^1 n_\kappa^2(x, s)\,dx+ \frac{8}{3}(t-s)\,.
 \end{equation*}
 The right hand side, by virtue of  
$0\leq n_\kappa(x, t)\leq S(x, t)$,  
is bounded  for any $s>0$ by 
 $$
  \int_0^1S(x, s)^2dx + \frac{8}{3}(t-s)<\infty \,.
 $$

Taking the limit $\kappa\to0$, we deduce that $n$ and $\Dx n$ lie in 
$L^2_{\rm loc}(Q_{(0, T]})$, and the limit $n$ 
is non-negative and satisfies  (\ref{we-}).
This proves that $n$ is indeed a weak solution of (\ref{ne-}), as claimed.
Moreover, from Lemma \ref{l:uss}  it follows that the limit $n$  has the universal upper bound
that for every $t>0$, 
$$
n\leq S(x,t) =  x+\frac{1-x}{t}+\frac{2}{\sqrt{t}} \quad\mbox{for a.e. $x\in(0,1)$}\,,
$$
and Lemma \ref{l:ec} implies that for almost every $(x,t)\in Q_{(0,T]}$, its slope 
has the  one-sided bound 
$$
\Dx n \geq -\frac{4}{t}\,.
$$ 
Finally, the limit function $n\in C((0,T],L^2)$, due to the following estimate.
With 
\[
\omega(h) = \sup_{s\le t\le s+h} \int_0^1 x^p |n(x,t)-n(x,s)|\,dx\,, 
\qquad C_s = \sup_{x\in[0,1]} S(x,s)\,,
\]
whenever $0<t-s<h$ is so small that $\alpha = \omega(h)^{1/(p+1)}<1$ we have 
\begin{align*}
\int_0^1|n(x,t)-n(x,s)|^2\,dx \le C_s^2\alpha + 
C_s\alpha^{-p}\int_\alpha^1 x^p |n(x,t)-n(x,s)|\,dx 
\le (C_s^2+C_s)\alpha\,.
\end{align*}
By consequence, the energy estimate \qref{en-} follows from the one for $n_\kappa$.
This finishes the proof of Theorem~\ref{t.exist}.
\hfill\(\Box\)

{\ifdraft
\vfil\pagebreak 
\fi}

\section{Finite time condensation}\label{s:bec}

The results of this section establish Theorem~\ref{t:cond}, demonstrating that loss of photons
is due to the generation of a nonzero flux at $x=0^+$, that such a flux persists if ever formed, 
and that photon loss does occur if the initial photon number exceeds the maximum attained
in steady state.

Throughout this section, we let $n$ be any global weak solution of \qref{ne-}.

\subsection{Formula for loss of photon number.}
First, we show how any possible decrease of photon number in time
is related to the nonvanishing of $n(0,t)^2=n(0^+,t)^2$, which is the formal
limit of the flux $J$ at the origin $x=0$.  
The following result implies part (i) of Theorem~\ref{t:cond}
in particular.
\begin{lem}\label{l:mass} For any fixed $t > s > 0$,  
\begin{equation}\label{mass+}
\int_{0^+}^1 n(x, t)\,dx = \int_{0^+}^1 n(x, s)\,dx
-\int_s^t n^2(0,\tau)\,d\tau\,.
\end{equation}
Moreover, for any $t>0$
\begin{equation}\label{decay}
\int_{0^+}^1 n(x, t)\,dx \leq \frac{1}{2} +\frac{1}{2t}+\frac{2}{\sqrt{t}}\,.
\end{equation}
\end{lem}
\begin{proof}
Integration of equation (\ref{ne}a) over $(x, 1) \times (s, t)$, using $J_\eps(1, t)=0$, gives
$$
\int_x^1 n_\epsilon(y, \tau)\,dy \Big|_{\tau=s}^{\tau=t}  = - \int_s^t(x^2\partial_x n_\epsilon(x, \tau) +n_\epsilon^2(x, \tau)-2xn_\epsilon(x,\tau) ) \,d\tau\,.
$$
Taking an average in $x$ over $(\epsilon, a)$, we find
\begin{equation} \label{i:limbd}
\avint_\epsilon^a\int_x^1 n_\epsilon(y,\tau) \,dy\,dx\Big|_s^t 
+\int_s^t \avint_\epsilon^a n_\epsilon^2\,dx\, d\tau
= 
\int_s^t \avint_{\epsilon}^a (2xn_\epsilon-x^2 \partial_x n_\epsilon)\,dx\, d\tau\,.
\end{equation}
For the first term on the left-hand side, integrating on $[x,1]=[x,a]\cup[a,1]$ we note
\[
\avint_\epsilon^a\int_x^1 n_\epsilon(y,\tau) \,dy\,dx
= \int_a^1 n_\eps(y,\tau)\,dy + R
\]
where, because $\eps\le x$,
\begin{align*}
R = \avint_\eps^a \int_x^a n_\eps(y,\tau)\,dy &\le 
    \left(\avint_\epsilon^a \int_\eps^a n_\epsilon(y,\tau)^2 \,dy\,dx\right)^{1/2} \left(\avint_\epsilon^a \int_\eps^a 1 \,dy\,dx\right)^{1/2} \\
   & \leq \left(\int_\epsilon^a n_\eps^2\,dy\right)^{1/2} \left({a-\epsilon}\right)^{1/2}\\
   & \leq C_s a^{1/2}\,,
\end{align*}
due to the energy estimate in Theorem~\ref{t:regeg}. 
The right-hand side in \qref{i:limbd} is bounded above by
\begin{align*}
 &\left( \int_s^t \avint_\epsilon^a (2n_\epsilon-x\partial_x n_\epsilon)^2 \,dx \,d\tau \right)^{1/2}
   \left(\int_s^t \avint_\epsilon^a x^2 \,dx\, d\tau \right)^{1/2} \\
   & \leq  \left(\int_s^t \avint_\epsilon^a 8n_\eps^2+2x^2(\partial_x n_\epsilon)^2 \,dx \,d\tau \right)^{1/2} \left((t-s)a^2\right)^{1/2}\\
   & \leq  
\left(\int_s^t \int_\epsilon^1 n_\eps^2+x^2(\partial_x n_\epsilon )^2 \,dx \,d\tau \right)^{1/2} 
\left(\frac{8(t-s)a^2}{a-\eps}\right)^{1/2}
\\
   & \leq C_{t,s} \left( \frac{a^2}{a-\epsilon}\right)^{1/2}\,.
\end{align*}
Passing to the limit $\epsilon \downarrow 0$ first, we have
$$
\int_a^1 n(x, \tau)\,dx\Big|_s^t +\int_s^t \avint_0^a n^2(x, \tau)\,dx\,d\tau = O(1)a^{1/2}\,.
$$
The desired equality follows from further taking $a\downarrow 0$.  Moreover, by virtue of $n\leq S$,  we have 
$$
\int_0^1 n(x, t)\,dx \leq \int_0^1 S(x, t)\,dx=\frac{1}{2} +\frac{1}{2t}+2 t^{-1/2},
$$
for any $t>0$. The proof is complete.
\end{proof}

Because $n$ is a classical solution of $\Dt n=\Dx J$ for $x,t>0$, 
by integration over $x\in[a,1]$, $\tau\in[s,t]$
we can infer that the loss term in \eqref{mass+} 
arises from the exit flux from the interval $[a,1]$
in the limit $a\to0$.
Thus the following result provides a precise sense in which
the flux $J(a,t)$ converges to $n^2(0,t)$ as $a\to0$.
\begin{cor}
Whenever $t>s>0$ we have 
\begin{equation}
\lim_{a\to 0^+} \int_s^t J(a,\tau)\,d\tau = \int_s^t n^2(0,\tau)\,d\tau.
\end{equation}
\end{cor}

\subsection{Persistence of condensate growth.}
Next we prove part (ii) of Theorem~\ref{t:cond}, 
showing that $n(0, t)$ once positive will remain positive for all time. More precisely:
\begin{lem}\label{shock} 
If $n(0, t^*)>0$ for some $t^*>0$, then
$n(0, t)\geq b(t)\hat x$ for  all $t>t^*$,
where
$$
b(t)= \left((1+t^*/4) e^{2(t-t^*)}-1 \right)^{-1},
\quad 
\hat x =\min\left\{ \frac{t^*n(0, t^*)}{4}, 1\right\}\,. 
$$
 \end{lem}
\begin{proof}  From $\D_xn(x, t^*)\geq -4/t^*$ we have
$$
n(x, t^*)\geq \left( n(0, t^*)- \frac{4x}{t^*} \right)_+ \ge \frac4{t^*}(\hat x-x)_+
=b(t^*)(\hat x-x)_+\,.
$$
Hence $
n(x, t^*)\geq \hat  n(x, t^*)
$
with
$$
\hat n(x,t)=b(t)(\hat x-x)_+\,.
$$
Note that $\hat n(1, t)=0$, and  for $0<x<\hat x$ we have
\[
 L[\hat n] :=\Dt \hat n- x^2\Dxx \hat n+2\hat n-2\hat n\Dx \hat n = (\hat x-x) (b'(t)+2b(t)+2b(t)^2) =0\,.
\]  
We claim   $n\ge\hat n$  on $(0, \hat x]$ for $t>t^*$.  Let $\eps>0$.
Substitution of
$
n=\hat n +v+\epsilon \Psi
$
into the equation $L[n]=0$ gives
$$
\hat L[v]:= \Dt v-x^2\Dxx v+2v -2v\Dx \hat n -2n\Dx v = -\eps \hat L[\Psi]\,. 
$$
Choosing
$
\Psi=-t+\log x,
$
we have $\Psi<0$ and
$$
\hat L[\Psi] = 2\Psi +2\Psi b - \frac {2n}x <0\,,
$$
hence $\hat L[v]>0$ for $0<x<\hat x$, $t\ge t^*$.
For  $0<x\le\sigma$ for $\sigma$ sufficiently small (depending on $\eps$),
$$
v =n-\hat n -\epsilon \Psi\ge -\hat n+ \epsilon (t-\log \sigma)>0\,, \quad\forall t>t^*\,.
$$
Moreover, at $x=\hat x$ we have
$$
v(\hat x, t)=n(\hat x, t)+\epsilon (t - \log \hat x) >0\,.
$$
 These facts, together with the fact $v(x, t^*)\geq \epsilon (t -\log x)> 0$, ensure that
 $$
 v(x, t) > 0\,, \quad \forall t>t^*, \ x\in (0,\hat x]\,.
 $$
Since $\eps>0$ is arbitrary, we infer  
 $$
 n(x, t)\geq \hat n(x,t)\,, \quad  \forall t \geq t^*, \ x\in(0, 1]\,.
 $$
This gives the desired estimate upon  taking $x\to0$.
\end{proof}

\subsection{Formation of condensates.}
The next result shows that photon loss will occur---meaning a condensate will form---in finite
time if the initial photon number $N[\nin] >\frac12$. This proves part (iii) of Theorem~\ref{t:cond}.  Note that $\frac12=N[x]$
is the maximum photon number for any steady state. 

\begin{prop} 
If 
$
N[\nin]>\frac{1}{2}\,,
$
then photon loss begins in finite time. More precisely, we have $n(0,t)>0$ whenever
\begin{equation}\label{i:tloss}
\frac{1}{2\sqrt t} < \sqrt{1+\delta}-1\,, \quad\mbox{where}\quad 2\delta = 
N[\nin]-\frac{1}{2}\,.
\end{equation}
\end{prop}
\begin{proof} From the supersolution  obtained in Lemma 3.2 it follows that
$$
n(x, t) \leq x + \frac{1-x}{t}+2 t^{-1/2}\,.
$$
Integration in $x$ over $(0, 1)$ leads to
$$
N[n(\cdot,t)] \leq \frac{1}{2} +\frac{1}{2t}+\frac{2}{\sqrt{t}}\,.
$$
Using Lemma~\ref{l:mass} we have
$$
\int_0^tn(0,\tau)^2\,d\tau \geq N[\nin] -\frac{1}{2}
-\frac{2}{\sqrt{t}}
- \frac{1}{2t} 
= 2\delta + 2-2\left(1+\frac1{2\sqrt t}\right)^2 \,.
$$
The right-hand side becomes positive when \qref{i:tloss} holds. The conclusion then
follows from Lemma~\ref{shock}.
\end{proof}

\subsection{Absence of condensates.}
Part (iv) of Theorem~\ref{t:cond} follows from a simple comparison:
If $\nin(x)\leq x$, then since $x=n_0(x)$ is a steady weak solution, the comparison
property from Theorem~\ref{t.uniq} implies  $n(x,t)\le x$ for all $t\ge0$.  Then $n(0^+,t)=0$,
so by part (a),
no condensate is formed
and we have
$$
N[n(\cdot,t)]=N[\nin] \quad \mbox{for all $t>0$}\,.
$$

\section{Large time convergence}\label{s:largetime}
We now investigate the large time behavior of solutions with non-trivial
initial data. In the system \qref{ne-}, the flux vanishes
for any equilibrium:
\begin{equation}
0=J = n^2\Dx\left(x-\frac{x^2}n\right)\,.
\end{equation}
Consequently $n=\nmu$ for some constant $\mu\ge0$, where
\[
\nmu(x) = \frac{x^2}{x+\mu}\,.
\]
Our main goal in this section is to prove
Theorem~\ref{t:lim}, which means that for every solution of \qref{ne-}
provided by
Theorem~\ref{t.exist} with nonzero initial data $\nin$,
there exists $\mu\geq 0$ such that
\begin{equation}\label{e:nmulim}
\|n(\cdot,t)-\nmu\|_1 = \int_0^1|n(x, t)-\nmu(x)|\,dx \to 0
\quad\text{as $t\to\infty$.}
\end{equation}

It will be convenient to denote by
\[
n(\cdot,t)=U(t)a
\]
the solution of \qref{ne-} with initial data $\nin(x)=a(x)$, $x\in(0,1)$.
 Due to the bound $n(x,t)\le  S(x,t)$
that holds by Theorem~\ref{t.exist}(i),  
for $t\ge1$ any solution $U(t)\nin$ will lie in the set
$$
A:=\{ a\in L^\infty(0,1): 0 \leq a(x) \leq 3
\ \text{ for a.e. $x\in(0,1)$}\}\,,
$$
since $S(x,1)\equiv3$.  The set $A$ is positively invariant under
the semi-flow induced by the solution operator:
\[
U(t)A\subset A\,, \quad t\ge0\,.
\]
With the metric induced by the $L^1$ norm,
$$
\rho(n_1, n_2) =\|n_1-n_2\|_1\,,
$$
the set $A$ is a complete metric space, and
by Lemma \ref{contraction}, $U(t)$ is a contraction: We have
\[
\|U(t)a-U(t)b\|_1\le \|a-b\|_1\,.
\]
whenever $t\ge0$ and $a$, $b\in A$.

For present purposes it is important that
a stronger contractivity property also holds, as shown in Lemma~\ref{contraction}:
Namely, if the functions $a$ and $b$ are $C^1$ and {\it cross
transversely}, then for $t>0$, $U(t)$ strictly contracts the $L^1$ distance
between $a$ and $b$.  Based on these contraction properties and the one-sided
Oleinik bound in Theorem~\ref{t.exist}(ii), 
we proceed to establish the large time convergence \qref{e:nmulim}.

We introduce the usual $\omega$-limit set of any element $a\in A$ as
$$
\omega(a) = \cap_{s>0}\,\overline{
\{U(t)a\mid t\ge s\}}\,.
$$
We have $b\in\omega(a)$ if and only if there is a sequence $\{t_j\}\to \infty$
such that $\|U(t_j)a- b\|_1\to 0$.

\begin{lem}\label{l:compact}(The $\omega$-limit set) Let $a\in A$.
Then $\omega(a)$ is not empty, and is invariant under $U(t)$, with
\begin{equation}\label{uo}
U(t)\omega(a)=\omega(a)\,, \quad t>0\,.
\end{equation}
Moreover, for any $b\in \omega(a)$, $b$ is smooth
(at least $C^2$ on $(0,1]$)
and satisfies
\begin{equation}\label{e:bbx}
\Dx b(x) \geq 0\,, \quad  0\leq b(x) \leq x\,,\qquad 0<x<1\,.
\end{equation}
\end{lem}
\begin{proof} To show $\omega(A)$ is not empty,
note that for any sequence $t_j\to\infty$, the estimates from
Lemma~\ref{l:bvlem} show that $\{U(t_j)a\}$ is bounded in $BV$.
By virtue of the Helley compactness theorem, some subsequence converges
in $L^1$, and this limit belongs to $\omega(a)$.

Next we prove (\ref{uo}). Given $b\in \omega(a)$, there exists $t_j$ such that
$$
\|U(t_j)a-b\|_1 \to 0\,, \quad j\to \infty\,.
$$
From $L^1$ contractivity and the semigroup property it follows that
$\|(U(t+t_j)a-U(t)b\|_1 \to 0$, hence $U(t)b\in \omega(a)$.
On the other hand, if $b\in U(t)\omega(a)$, we have $b=U(t)b^*$
with $b^* \in \omega(a)$. Then for some sequence $t_j\to\infty$,
$$
\|U(t+t_j)a-b\|_1=\|U(t)U(t_j)a-U(t)b^*\|_1\leq \|U(t_j)a-b^*\|_1 \to 0
$$
as $j\to \infty$, hence  $b\in \omega(a)$.

By relation (\ref{uo}), for each $b\in\omega(a)$ and $t>0$, $b=U(t)b^*$ for
some $b^*\in \omega(a)$. From this it follows $b$ is smooth
and that $\Dx b\ge-4/t$ and $0\le b(x)\le S(x,t)$ by Theorem~\ref{t.exist}.
Taking $t\to\infty$, since $S(x,t)\to x$ we infer \qref{e:bbx}.
\end{proof}

\begin{lem}(Equilibria and $\omega(a)$)
(i) If  $n_\mu\in \omega(a)$ for some $\mu \geq 0$, then
\begin{equation}\label{ma}
\lim_{t\to \infty} \|U(t)a -n_\mu\|_1 =0\,.
\end{equation}
(ii) Let $b\in \omega(A)$. Then for any $\mu \geq 0$,
\begin{equation}\label{dmu}
 \|b-n_\mu\|_1 =\|U(t)b - n_\mu \|_1\,.
 \end{equation}
(iii) If $a\not\equiv0$, then $0\notin\omega(a)$.
\end{lem}
\begin{proof} (i) By definition, for any $\epsilon>0$,  $\|U(t_j)a-n_\mu\|_1<\epsilon$ for large $t_j$.  This ensures that for any $t>t_j$, 
$$
\|U(t)a-n_\mu\|_1=\|U(t-t_j)U(t_j)a-U(t-t_j)n_\mu\| \leq\|U(t_j)a-n_\mu\|_1<\epsilon \,,
$$
hence (\ref{ma}). \\
(ii) Since  $b\in \omega(A)$, there exists $a\in A$ and a sequence $\{t_j\}$ such that $t_j\to \infty$ as $j\to \infty$ and
$$
\lim_{t \to \infty} \|U(t_j)a-b\|_1=0\,.
$$
Given any $\mu \geq 0$, by contraction of $U(t)$ we know that
$$
\|U(t)a -n_\mu\|_1 = \|U(t)a - U(t)n_\mu\|_1
$$
is decreasing in time and thus admits a limit $c_\mu \geq 0$ as $t \to \infty$, i.e.,
$$
\lim_{t\to \infty}  \|U(t)a-n_\mu\|_1=c_\mu\,, \quad  t \to \infty\,.
$$
Letting $t=t_j$  in the above equation and passing to the limit,  we have
$$
\|b- n_\mu \|_1=c_\mu\,.
$$
Note that if $b \in \omega(a)$, then $U(t)b \in \omega(a)$; thereby
$$
 \|U(t)b -n_\mu\|_1 =c_\mu\,.
$$
Therefore (\ref{dmu}) holds  for $\forall t>0, \quad  \mu \geq 0$.

(iii) Suppose $a\not\equiv0$, so that $N[a]>0$.
We claim $0\notin\omega(a)$.  Supposing $0\in\omega(a)$ instead,
we write $n(\cdot,t)=U(t)a$.
Then $N[n(\cdot,t)]=\|U(t)a-0\|_1$ is non-increasing and approaches zero
as $t\to\infty$. 
By Lemmas~\ref{l:mass} and \ref{shock}, then,
a condensate forms and $n(0^+,t)>0$ for all large $t$.

From the Oleinik-type lower bound of Theorem~\ref{t.exist}(ii),
 $x<z<1$ entails $n(x,t)-\frac4t\le n(z,t)$. After integration
from $1-x$ to $1$ we find
\[
 (1-x)\left(n(x,t)-\frac4t\right)\le N[n(\cdot,t)]\,.
\] 
For $t$ large enough we have $N[n(\cdot,t)]<\frac1{16}$ and $t>32$,
and this ensures that 
\[
\mbox{for all $x\in[\frac14,\frac12]$}\,,\qquad 
 n(x,t) \le 2N[n(\cdot,t)] + \frac4t <\frac14\le x\,.
 \]
Then, because $n(0^+,t)>0$,  the last crossing point defined by
\begin{equation}
x_1 = \max\{ x\in(0,\frac14]: n(x,t)=x\}
\end{equation}
is well defined. Using again Theorem~\ref{t.exist}(ii), it now follows
\begin{align*}
0\le n(x,t)\le x_1+ \frac4t x_1 \quad \mbox{for $0<x<x_1$}\,,   \\
x \ge n(x,t)\ge x_1-\frac4t x_1 \quad \mbox{for $x_1<x<2x_1$}\,.
\end{align*}
From these inequalities, we deduce respectively that
\begin{align*}
\int_0^{x_1}|x-n(x,t)|\,dx &\le x_1^2\left(1+\frac 4t\right)\,,\\
\int_{x_1}^{2x_1}|x-n(x,t)|\,dx &\le \int_{x_1}^{2x_1} x\,dx - x_1^2\left(1-\frac4t\right)\,.
\end{align*}
We may also assume $t$ is so large that $S(x,t)<2x$ for
$\frac12\le x\le 1$. Then since $0\le n(x,t)\le S(x,t)$,
it follows
\[
\int_{2x_1}^1 |x-n(x,t)|\,dx \le \int_{2x_1}^1 x\,dx\,.
\]
Because $x_1^2 (8/t)<\int_0^{x_1}x\,dx$, after
adding the last three inequalities we find
$\|x-U(t)a\|_1< \|x-0\|_1$.
But then since $\|x-U(t)a\|_1$
is nonincreasing in $t$, it is impossible that $\|U(t)a-0\|_1\to0$
as $t\to\infty$. This proves $0\notin\omega(a)$.
 \end{proof}

The following restatement of the result in Lemma \ref{contraction}
 plays a critical role in proving (\ref{e:nmulim}).
\begin{lem} \label{l:tran} If $ a, b \in A\cap C^1((0,1))$
and $a$ and $b$ cross transversely at least once on $(0, 1)$, then
$$
\|U(t)a- U(t)b \|_1 <\|a-b\|_1\,, \quad t>0\,.
$$
\end{lem}

We are now ready to prove (\ref{e:nmulim}).
Let $a\in A$ with $a\not\equiv0$.  By Lemma \ref{l:compact} we know that
$\omega(a)$  is not empty.  Let $b\in \omega(a)$.
We need to show there exists a $\mu \geq 0 $ such that
 \begin{equation}\label{bnmu}
 b = n_\mu\,.
\end{equation}
Since  $b\not\equiv0$ and $b$ is non-decreasing, the 
quantity
\[
g(x)=x-\frac{x^2}{b(x)} \,,
\]
which is the first variation $\delta H/\delta n$ of entropy,
is well defined in some non-empty interval $(x_0, 1)$.
If $g$ is not a constant, there exists  some $x^*\in (x_0, 1)$
such that $g'(x^*)\not=0$.
Then it follows that at $x=x^*$,
with $\mu^*=-g(x^*)$ we have
$$
b=\frac{x^2}{x-g(x)} = n_{\mu^*}\,, \qquad 
\Dx b = \Dx n_{\mu*}+\frac{x^2 g'}{(x-g(x))^2}\not= \partial_x n_{\mu^*}\,.
$$
In other words, $b$ and $n_{\mu^*}$ cross transversely at
$x^*$. Therefore by Lemma \ref{l:tran}  we have
$$
\|U(t)b - U(t) n_{\mu^*} \|_1 < \|b-n_{\mu^*}\|_1\,.
$$
This contradicts (\ref{dmu}).  We conclude that $g$ must be a constant,
i.e.,  $g(x)=-\mu$, which gives  (\ref{bnmu}).   From $b\not=0$ and $b\leq x$ we
see that $\mu \geq 0$.

\begin{rem} Due to loss of mass,  determining  $\mu$ quantitively for each given initial data is not straightforward, except for some special cases as treated in Corollary~\ref{c:mulim}.
\end{rem}
\begin{proof}[Proof of Corollary~\ref{c:mulim}]
If $\nin \geq x$, by the comparison  result in Theorem~\ref{t.uniq},  we have
$$
x \leq n(x, t)\,, \quad t>0\,.
$$
On the other hand, the supersolution bound from Theorem~\ref{t.exist}(i) ensures that
$$
n(x, t) \leq x+\frac{1-x}{t}+2t^{-1/2}.
$$
These together lead to (\ref{nd:f}), hence  $\lim_{t \to \infty} n(x, t)= x$.

In the case of $\nin\leq x$,  we have $n(x,t)\le x$ for all $t$. 
Then there is no mass loss, hence the limiting equilibrium state $n_\mu$ 
satisfies
\[
\int_0^1 n_\mu\,dx = \int_0^1 \nin\,dx =N[\nin]\ .
\]
Integration of the left-hand side yields \qref{mu}.
\end{proof}

\appendix

\section{Anisotropic Sobolev estimates.}

For use in section 4 and Appendices~\ref{s:regexist} and~\ref{s:interior}, we need some
basic anisotropic Sobolev estimates that are not easy to find in the extensive literature on the subject. 
The results that we need appear to be related to  embedding results for 
anisotropic Besov spaces contained in the books \cite{BIN12}.  
For the reader's convenience, however, we provide a self-contained treatment
based on simple estimates for Fourier transforms.

If $\Omega\subset \R^2$,  the typical anisotropic Sobolev space is
$$
u\in W^{2k, k}_2(\Omega)=\{u\mid D^s_x D^r_t u \in L^2(\Omega), \ 0\leq 2r+s\leq 2k\}.
$$
As usual,  if a function $u\in  W^{2k, k}_2(\Omega)$,  it will automatically belong to certain other spaces, which depend on $k$ and the dimension. One such space is $C^{\gamma, \gamma/2}(\Omega)$. We say $u\in C^{\gamma,\gamma/2}(\Omega)$ if there is a constant $K$ such that
$$
|u(x, t)- u(y, \tau)|\leq K (|x-y|^2+|t-\tau|)^{\gamma/2} 	\qquad\forall (x, t), (y, \tau)\in \Omega.
$$
The space $C^{\gamma, \gamma/2}(\Omega)$ is a Banach space with norm given by 
$$
\|u\|_{C^{\gamma, \gamma/2}(\Omega)}=\max_{(x, t)\in \Omega}|u(x, t)|+\sup_{(x, t), (y, \tau)\in \Omega}\frac{|u(x, t)- u(y, \tau)|}{(|x-y|^2+|t-\tau|)^{\gamma/2}}.
$$
The results we need are contained in the following result.
\begin{thm}\label{t:ani}
Let $D= (a,b)\times(c,d)$ be a rectangular domain in $\R^2$.
Suppose that $u$ and its distributional derivatives $\Dt u$ and $\Dx^2 u$ 
are in $L^2(D)$, i.e., $u\in W_2^{2, 1}(D)$.  Then $u\in C^{1/2,1/4}(D)$,  and there is a constant $C$ depending on $D$ and $s$ such that 
$$
\|\Dx u\|_{L^s(D)} \leq C \|u\|_{W_2^{2, 1}(D)}, \quad 2\leq s<6.  
$$
\end{thm}
 With very little more work, one can discuss higher-order embeddings and arbitrary space dimensions.  We will not provide the details here, but the result is summarized as follows. 
\begin{thm} Let $D$ be a bounded parabolic cylinder in $\R^{n+1}$ with $C^1$ spatial boundary.  Suppose that $u\in W_2^{2k, k}(D)$, then 
\begin{flalign*}
\text{(i)} & \qquad
W_2^{2k, k}\to \left\{
 \begin{array}{lll} C^{\gamma, \gamma/2}, & \gamma=2k-\frac{n+2}{2}, & k>\frac{n+2}{4},\\
L^s, & 2\leq s <\infty, & k=\frac{n+2}{4}, \\
 L^s, & 2\leq s< \frac{2(n+2)}{(n+2)- 4k}, &k<\frac{n+2}{4}.
 \end{array}
\right.
&
\\
(ii) &\qquad
C\|u\|_{W_2^{2k, k}} \geq  \left\{
 \begin{array}{lll}
 \|\Dx u\|_\infty, & {}  & k>\frac{n}{4} +1,\\
\|\Dx u\|_{L^s(D)}, & 2\leq s <\infty, & k=\frac{n}{4} +1, \\
 \|\Dx u\|_{L^s(D)}, & 2\leq s< \frac{2(n+2)}{(n+4)-4k}, &k<\frac{n}{4} +1.
 \end{array}
\right.
&
\end{flalign*}
\end{thm} 
 The results in Theorem B.1 correspond to $n=1$, $k=1$ and the cases $\gamma=1/2$ 
 in part (i) and $s\in [2, 6)$ in part (ii).

\subsection{Fourier estimates in $\R^2$} 
The Fourier transform for $u\in L^1(\mathbb{R}^{2})$ is 
\begin{equation}\label{f2}
\hat u(\xi, l)=\int_{\mathbb{R}^{2}} u(x, t) e^{-2\pi i( x\xi +t l)}dx\,dt,
\end{equation}
which extends to a bounded linear map $ u\to \hat u$ from $L^p$ to $L^{p'}$, for $1\leq p\leq 2$ and $1/p+1/p'=1$. Moreover, the Hausdorff-Young inequality holds:
\begin{equation}\label{hy}
\|\hat u\|_{p'} \leq  \|u\|_p
\end{equation}
for $u\in L^p$.  This  simply interpolates  $\|\hat u \|_\infty \leq \|u\|_1$ and the Plancherel theorem, $\|\hat u\|_2=\|u\|_2$. The continuity of $\hat u$ follows from the dominated convergence theorem. 
In case $\hat u$ is integrable, one may recover $u$ from $\hat u$ by
\begin{equation}\label{inf2}
u(x, t)= \int_{\mathbb{R}^{2}} \hat u(\xi, l) e^{2\pi i( x \xi +tl)}d\xi dl.
\end{equation}
We will deduce Theorem \ref{t:ani} from the corresponding result on all of $\R^2$:
\begin{thm} \label{t:anii} Suppose $u\in W^{2, 1}_2(\mathbb R^{2})$. Then 
 $u\in C^{1/2, 1/4}(\mathbb{R}^2)$.
Moreover, $\Dx u\in L^s(\mathbb R^{2})$ for $2\leq s<6$, with 
$$
 \|\Dx u\|_{L^s(\mathbb{R}^2)} \leq C\|u\|_{W_2^{2, 1}(\mathbb{R}^2)}. 
 $$
 \end{thm}

To proceed, we first recall  a characterization of  $W^{2k, k}_2$.
\begin{lem}(Characterization  of $W^{2k, k}_2(\mathbb{R}^{2})$ by Fourier transform).  Let $k$ be a nonnegative integer, and set
$$
m(\xi,l) = (1+l^2+|\xi|^4)^{1/2}.
$$
Then $u\in W^{2k, k}_2(\mathbb{R}^{2})$ if and only if
$
m^{k} \hat u\in L^2(\mathbb{R}^{2}).
$
In addition, there exists a constant $C$ such that
$$
C^{-1} \|u\|_{W^{2k, k}_2} \leq \|m^{k} \hat u\|_{L^2}  \leq C \|u\|_{W^{2k, k}_2}.
$$
\end{lem}
The following two technical lemmas will be used as well.
\begin{lem} \label{Aab} For $0\leq \alpha < 2\beta $ and $\beta \geq 1$, we have
$$
A_{\alpha, \beta}: = |\xi|^\alpha/m^\beta \in L^s(\mathbb{R}^{2})
\qquad\mbox{
if any only if}\qquad
s > \max\left\{ \frac{3}{2\beta-\alpha}, \frac{1}{\beta} \right\}.
$$
\end{lem}
\begin{proof} 
A direct calculation using the substitution $l=y(1+|\xi|^4)^{1/2}$
gives
\begin{align*}
\|A_{\alpha, \beta}\|_s^s  &  = \int_{\mathbb{R}^2}  \frac{|\xi|^{\alpha s}}
{ (1+l^2+|\xi|^4)^{\beta s/2}}d\xi\, dl \\
 & = \int_{\mathbb{R}^2}  \frac{|\xi|^{\alpha s} (1+|\xi|^4)^{1/2}}{(1+y^2)^{\beta s/2} (1+|\xi|^4)^{\beta s/2}} d\xi\, dy\\
 &= \int_{\mathbb{R}} \frac{dy}{(1+y^2)^{\beta s/2}} 
 \int_\R \frac{|\xi|^{\alpha s} (1+|\xi|^4)^{1/2}}{(1+|\xi|^4)^{\beta s/2}} 
 d\xi.
\end{align*}
This is bounded if and only if $\beta s>1$ and
$
2\beta s -2-\alpha s >1.
$
That is, $s\beta >1$ and  $s(2 \beta-\alpha)>3$.
\end{proof}
\begin{lem} \label{lem:gr} 
Let $V(x)=|x|\wedge1:=\min\{|x|,1\}$. Then 
for some constant $C>0$,
\begin{align}\label{gr}
\|m^{-1} V(r\xi)\|_2 + \|m^{-1} V(r^2l)\|_2  \leq Cr^{1/2}
\quad\mbox{for all $r>0$.}
\end{align}
\end{lem}
\begin{proof} For the first term,  substituting $l=y(1+|\xi|^4)^{1/2}$ again,
we find
\begin{align*}
\|m^{-1} V(r\xi)\|_2^2  &   =\int_{\mathbb{R}^2}  \frac{(|r \xi|\wedge 1)^2 }{(1+l^2+|\xi|^4)}d\xi \,dl \\
& \leq  \int_{\mathbb{R}^2} \frac{(|r \xi|\wedge 1)^{2} (1+|\xi|^4)^{1/2}}{(1+y^2) (1+|\xi|^4)} d\xi \,dy
\\
 &=  \int_{\mathbb{R}} \frac{d y }{(1+y^2)} \int_\R \frac{(|r\xi|\wedge 1)^{2} }{(1+|\xi|^4)^{1/2}} d\xi.
\end{align*}
The first factor is finite.  We proceed to decompose the last integral into two parts, one over $\{\xi: \;  |\xi|<r^{-1}\}$ and the other over $\{\xi:  \; |\xi|>r^{-1}\}$: The integrand is even, and
\begin{align*}
\int_0^\infty \frac{(|r\xi|\wedge 1)^{2} }{(1+|\xi|^4)^{1/2}} d\xi
&= \int_0^{r^{-1}} \frac{(r|\xi|)^{2}} {(1+|\xi|^4)^{1/2}} d\xi+ 
\int^\infty_{r^{-1}}  \frac{1} {(1+|\xi|^4)^{1/2}}d\xi\\
& \leq r^{2} \int_0^{r^{-1}} d\xi +  \int^\infty_{r^{-1}} |\xi|^{-2}d\xi=2r. 
\end{align*}
In a similar fashion, we estimate, using $\xi=(1+l^2)^{1/4}\eta$,
 \begin{align*}
\|m^{-1}V(r^2l)\|_2^2  &   =
\int_{\mathbb{R}^2}  \frac{| |r^2 l|\wedge 1|^2 }{1+l^2+|\xi|^4}d\xi\, dl \\
& \leq  \int_{\mathbb{R}^2}  \frac{(|r^2 l|\wedge 1)^{2}}{(1+l^2)^{3/4} (1+|\eta|^4)} d\eta\, dl\\
 &=  \int_\R \frac{d\eta }{(1+|\eta|^4)}  \int_\R \frac{(|r^2 l|\wedge 1)^{2}}{(1+|l|^2)^{3/4}} dl.
\end{align*}
The first integral is bounded; the second integral is further estimated by
\begin{align*}
\int_0^\infty \frac{(|r^2 l|\wedge 1)^{2}}{(1+|l|^2)^{3/4}} dl
&\leq \int_0^{r^{-2}} \frac{(r^2|l|)^{2}} {(1+|l|^2)^{3/4}} dl
+ \int^\infty_{r^{-2}}  \frac{1} {(1+|l|^2)^{3/4}}dl\\
& \leq  r^4 \int_0^{r^{-2}} |l|^{1/2}dl +  \int^\infty_{r^{-2}} |l|^{-3/2}dl\\
& =\left(\frac{2}{3}+2\right)r=\frac{8}{3}r.
\end{align*}
These estimates together yield the bound (\ref{gr}) as claimed.
\end{proof}

\begin{proof}[Proof of Theorem~\ref{t:anii}] From the inversion formula (\ref{inf2}) it follows that 
\begin{align*}
\|u\|_\infty & \leq \| \hat u\|_1 
 \leq \|m \hat u\|_{2} \|m\inv\|_2
 \leq C \|u\|_{W^{2, 1}_2},
\end{align*}
where the bound on $\|m\inv\|_2=\|A_{0, 1}\|_2$ is ensured by Lemma \ref{Aab}. 

(i) Fix  $(x, t)\not=(y, \tau)$ so that $r=\sqrt{|y-x|^2+|\tau -t|}>0$.
Using the inequalities
$$
|e^{2ia}-e^{2ib}| \leq  2|a-b|\wedge  2  = 2V(a-b) ,
$$
\[
|(y-x)\cdot \xi +(\tau-t)l| \le r|\xi|+r^2|l|,
\]
we obtain from the inversion formula and  Lemma \ref{lem:gr} that
\begin{align*}
|u(x, t)-u(y, \tau)| & \leq 2 \pi \int ( V(r\xi)+V(r^2l)) |\hat u(\xi, l)| d\xi\, dl    \\
& \leq 2\pi  ( \|m^{-1}V(r\xi)\|_2 + \|m^{-1}V(r^2l)\|_2) \|m\hat u\|_2\\
& \leq Cr^{1/2} \|u\|_{W^{2, 1}_2}.
\end{align*}
This proves the embedding $W_2^{2, 1}(\mathbb{R}^2) \to  C^{1/2, 1/4}(\mathbb{R}^2)$.

(ii)  For $2\leq s<6$ we have $s'=\frac{s}{s-1}\leq 2$ and $r>3$ where
\[
\frac{1}{r} =\frac{1}{2}- \frac{1}{s}.
\]
We may then use the Hausdorff-Young inequality (\ref{hy}) and Lemma~\ref{Aab} to obtain 
\begin{align*}
\|\Dx u\|_s & \leq  C_s \| \xi\hat u\|_{s'} \leq \|m\hat u\|_{2} \| A_{1, 1} \|_r, 
  \leq C \|u\|_{W^{2, 1}_2}.
\end{align*}
\end{proof}

\begin{proof}[Proof of Theorem \ref{t:ani}]
Let $D$ be the given closed rectangle in $\mathbb{R}^2$. For functions $u$ defined a.e. on $D$, we extend $u$ from $D$ to a larger rectangle $\hat D$
containing $D$ in its interior, in two steps, using linear combinations of
dilated reflections as shown in Adams \cite[Theorem 4.26]{Ad75}.  The extension is to be made so that the weak derivatives are preserved across $\partial D$. 

For instance, we reflect across the faces of $D$ sequentially: First,  from $x$ faces $\{a, b\}$ with $c\leq t\leq d$, writing $\hat a=a-(b-a)$, $\hat b=b+(b-a)$, let
$$
E_xu(x, t)=\left\{ \begin{array}{ll}-3u(2a-x, t)+4u(-x/2+3a/2), & \quad \hat a \leq x \leq a,\\
u(x, t), & \quad a\leq x\leq b, \\
-3u(2b-x, t)+4u(-x/2+3b/2), & \quad b\leq x\leq \hat b.
\end{array}
\right.
$$
and then from the $t$ faces $\{c, d\}$ in an entirely similar manner, such that 
$$
\tilde u = E_{t} E_x u(x, t)
$$
is an $C^1$ extension when crossing $\partial D$ and well-defined in 
$\hat D=[\hat a,\hat b]\times[\hat c,\hat d]$.
Then multiply by a fixed smooth cutoff function $\phi(x, t)$  that is $1$ on $D$ and $0$ near the  boundary of $\hat D$ to obtain 
$$
Eu=\phi(x, t)E_tE_x u(x, t).
$$
In this way, given $u$ such that $\Dt u$ and $\Dx^ju$ are in $L^2(D)$ for $j=0,1,2$, we obtain $Eu$ such that $\Dt Eu$
and $\Dx^jEu$ are in $L^2(\R^2)$ for $j=0,1,2$. The extension $E$ is thus a bounded linear operator from  $\in W^{2, 1}_2(D)$ 
to  $ W^{2, 1}_2(\mathbb{R}^{2})$. Moreover,   
$$
Eu=u \quad {\rm a.e. \quad in}\quad D, 
$$
$Eu$ has support in $\hat D$,  and 
$$
 \|Eu\|_{W^{2, 1}_2(\mathbb{R}^{2})}\leq C\|u\|_{W^{2, 1}_2(D)}.
$$
This combined with Theorem \ref{t:anii} when applied to $Eu$ proves Theorem \ref{t:ani}. 
\end{proof}

\section{Existence for the truncated problem}\label{s:regexist}
In this appendix, we establish the existence of a classical solution
to the truncated problem \qref{ne}.  We aim to prove Proposition~\ref{t:regexist}.
This global existence result does not appear to follow easily from stated results
in standard parabolic theories,  due to the fact that the boundary condition 
$J=0$ at $x=1$ is nonlinear. For the convenience of the reader, we indicate how
to establish Theorem~\ref{t:regexist} by use of  an approximation method that
involves cutting off the nonlinear term in the flux $J$
together with interior regularity theory. 
This will result in a problem
with standard linear Robin-type boundary conditions, that still respects
a maximum principle which keeps the solution uniformly bounded.

\subsection{Approximation by flux cut-off.}
We consider, then, the following problem. 
Let $\chi(x)$ be a smooth, nondecreasing function
with $\chi(x)=0$ for $x<-1$, $\chi(x)=1$ for $x>1$ as in \qref{d:rhochi}.
For small $h>0$ define $\chi_h(x) = \chi( 1+(x-1)/h) $, so that
\begin{equation}\label{d:chih}
\chi_h(x) = \begin{cases}
0\,, &  x<1-2h\,, \cr
1\,, & x=1\,.
\end{cases}\,.
\end{equation}
Writing 
\begin{equation}\label{d:Jh}
J_h = x^2\Dx \neph -2x\neph +  \neph^2 + (3\neph-\neph^2)\chi_h\ ,
\end{equation}
we consider the problem
 \begin{subequations}\label{neh}
\begin{align}
 \Dt \neph  &= \Dx J_h\,,
&\qquad x\in (\epsilon, 1), \ t\in(0,\infty),
\label{e:nep1h}
\\
 \neph &=n\init_h, &\qquad x\in (\epsilon, 1), \ t=0,
\label{e:nep2h}
\\
0&=J_h \,, 
&x=1,\ t\in[0,\infty),
\label{e:nep3h}
\\
0&= \eps^2\Dx\neph-2\eps\neph,
&x=\eps,\ t\in[0,\infty)\,.
\label{e:nep4h}
\end{align}
\end{subequations}
We construct the initial data $\nin_h$ from the given $\nin_\kappa$ 
so that at $t=0$, the cut-off flux $J_h$ is the original $J_\eps$. Namely,
we require that at $t=0$,
\begin{equation}\label{ic:Jh}
J_h = x^2 \Dx\nin_\kappa -2x\nin_\kappa+ (\nin_\kappa)^2 \ .
\end{equation}
We make $\nin_h(x)=\nin_\kappa(x)$ for $x<1-2h$, and use
\eqref{ic:Jh} to determine $\nin_h(x)$ for $x\in[1-2h,1]$.
These initial data are compatible with the boundary conditions \qref{e:nep3h}--\qref{e:nep4h}.
Clearly in the limit $h\downarrow0$, we have $\nin_h\to\nin_\kappa$
uniformly on $[\eps,1]$. 

\subsection{Uniform bounds on the cut-off problem.}
Because $\chi_h(1)=1$, the boundary condition in \qref{e:nep3h} is linear in $\neph$,
taking the form 
\begin{equation}
\Dx\neph=-\neph \qquad x=1, \ t\in[0,\infty)\ .
\end{equation}
Moreover, note that \qref{e:nep1h} takes the  explicit form
\begin{equation}
\Dt \neph = x^2\Dx^2\neph -2\neph +
(\Dx\neph)(2\neph+(3-2\neph)\chi_h) + (3\neph-\neph^2)\chi_h' \ .
\label{e:nepexpl}
\end{equation}
For this problem, comparison principles hold, whence we obtain 
positivity and uniform sup-norm bounds on solutions.

\begin{lem}\label{l:max-h}
Suppose $\min_{[\eps,1]}\nin_h>0$ and $\max_{[\eps,1]}\nin_h < M_1$ where $M_1\ge 3$.
Suppose $\neph$ is a classical solution of \qref{neh} in $[\eps,1]\times(0,T]$,
with $\neph$ continuous on $[\eps,1]\times[0,T]$. Then $0<\neph(x,t)<M_1$ for all $(x,t)\in[\eps,1]\times[0,T]$. 
\end{lem}

\begin{proof}
The proof of strict positivity is similar to Lemma 4.2. 
To prove the upper bound, suppose $n_h(X^*)=M$ with $X^*=(x^*,t^*)$, where $t^*>0$ is minimal.
Because $\Dx\neph=-\neph<0$ holds at $x=1$, and \qref{e:nep4h} holds at $x=\eps$,
$x^*$ must lie strictly between $\eps$ and $1$. But because \qref{e:nepexpl} holds
and $\chi_h'\ge 0$, this is impossible.
\end{proof}

We may obtain global existence of a classical solution to problem \qref{neh}
with cut-off flux from the proof of Proposition~7.3.6 of \cite{Lunardi}, 
due to the time-uniform bounds on $\neph$ in this Lemma, and the fact that 
the nonlinear terms in \qref{e:nep1h} appear in the divergence form $N_h(\neph):=\Dx(\neph^2(1-\chi_h))$,
which enjoys a local Lipschitz bound in the $L^\infty$ norm of the form
\begin{equation}
\|N_h(u)-N_h(v)\|_\infty\le 
K\Bigl( \|u-v\|_\infty\|u\|_{C^1} + \|v\|_\infty \|u-v\|_{C^1} \Bigr)\,,
\end{equation}
with $K=1+\|\chi_h'\|_\infty$. 

From the proof of \cite[Prop.~7.3.6]{Lunardi}, this solution $n_h$ is continuous
on  $[\eps,1]\times[0,\infty)=\bar Q^\eps$, and the quantities $\Dx n_h$, $\Dt n_h$ and  $\Dxx n_h$
are continuous on $[\eps,1]\times(0,\infty)$.  However, these 
quantities are actually all continuous on $\bar Q^\eps$ by
the local-time existence theorem~8.5.4 of \cite{Lunardi}, 
due to the fact that the initial data are $C^3$ and satisfy the 
compatibility conditions.  (A simple energy estimate for the 
difference, along the lines of step 1 in subsection \ref{sB:eest} below,
shows that the local solution given by this theorem
agrees with that given by Prop.~7.3.6.) 

Additionally, these quantities are also locally \holder-continuous
on $[\eps,1]\times(0,\infty)$, due to the regularity results 
stated in \cite[Prop.~7.3.3(iii)]{Lunardi}.
From standard interior regularity theory for parabolic problems %
(e.g., based on Theorem~8.12.1 of \cite{Krylov} and bootstrapping),
we infer that $n_h$ is smooth in $Q^\eps$. In particular,
the flux $J_h$ is a classical solution in $Q^\eps$ of the equation
\begin{equation} \label{e:Jht}
\Dt J_h = x^2\Dx^2 J_h + (\Dx J_h)(-2x + 2\neph + (3-2\neph)\chi_h) \,.
\end{equation}
Since 
 $J_h$ is continuous on $\bar Q^\eps$, by the maximum
principle it is bounded in terms of its initial and boundary values---recall
 $J_h =n_h^2$ at $x=\eps$.
From this and the sup-norm bound in the previous Lemma, we obtain
($\eps$-dependent) uniform bounds on $\Dx n_h$.
\begin{lem} \label{l:nxbdh} There is a constant $M_2$ depending on 
$\nin_\kappa$ and independent of $h$ and $\eps$, such that
$|J_h|+ \eps^2|\Dx n_h| \le M_2$ for all $(x,t)\in \bar Q^\eps$,
\end{lem}

\subsection{Energy estimates.}\label{sB:eest} These are simpler than the corresponding ones
in section 5, because here $\eps>0$ is fixed, and the initial data is smooth. 

1. The basic energy estimate is (using that $n_h$ is positive and bounded)
\begin{align*}
&\frac{d}{dt}  \int_\eps^1 \frac12 n_h^2\,dx  = \int_\eps^1 n_h\,\Dx J_h\,dx = 
 n_hJ_h\Bigr|_\eps^1 - \int_\eps^1 (\Dx n_h)J_h\,dx
\\
&\quad = -n_h(\eps,t)^3 
-\int_\eps^1
(\Dx n_h)(x^2\Dx\neph -2x\neph +  \neph^2 + (3\neph-\neph^2)\chi_h)\,dx
\\ &\quad \le -\frac{\eps^2}2 \int_\eps^1 (\Dx n_h)^2\,dx + C \int_\eps^1 n_h^2\,dx
\end{align*}
Here $C$ is independent of $h$ and $t$, and after integration 
 we conclude that $\Dx n_h$ (and also $J_h$) is uniformly bounded
independent of $h$ in $L^2$ on $[\eps,1]\times[0,T]$, for any $T$.

2. For $(x,t)\in Q^\eps$, the flux $J_h$ satisfies \qref{e:Jht}, and we find
\begin{align*}
\frac{d}{dt}  \int_\eps^1 \frac12 J_h^2\,dx 
&= J_h(x^2\Dx J_h)\Bigr|_\eps^1 - \int_\eps^1 (x\Dx J_h)^2\,dx
\\ & \qquad
+ \int_\eps^1 J_h(\Dx J_h)(-4x + 2\neph + (3-2\neph)\chi_h)
\\\quad &\le -\frac{\eps^2}3 \D_t(n_h(\eps,t)^3) 
- \frac{\eps^2}2 \int_\eps^1 (\Dx J_h)^2\,dx
+ C\int_\eps^1 J_h^2\,dx \,.
\end{align*}
Upon integration in time, we conclude $\Dx J_h = \Dt n_h$
is uniformly bounded
independent of $h$ in $L^2$ on $[\eps,1]\times[0,T]$, for any given $T$.
And further, using \qref{d:Jh} for $x<1-2h$, we deduce that 
$\Dx^2\neph$ is uniformly bounded independent of $h$ in $L^2$
on any compact set
\begin{equation}\label{c:set1}
[\eps,1-\hat\eps]\times[0,T] \subset [\eps,1)\times[0,\infty)
\end{equation}
fixed independent of $h$. 
(This does not work for $\hat\eps=0$ because $\chi_h'$ is not uniformly bounded.)

3. Next, we have
\begin{align*}
&\frac{d}{dt}   \int_\eps^1 \frac12 (\Dx J_h)^2\,dx = 
(\Dx J_h)(\Dt J_h)\Bigr|_\eps^1 
- \int_\eps^1 (\Dx^2J_h)(\Dt J_h)\,dx
\\ &\qquad \le -2n_h(\eps,t)(\Dt n_h(\eps,t))^2
- \frac{\eps^2}2 \int_\eps^1 (\Dx^2 J_h)^2 \,dx + C \int_\eps^1 (\Dx J_h)^2\,dx\,.
\end{align*}
Because $\Dx J_h$ is continuous on $\bar Q^\eps$, we may integrate 
this inequality over $t\in[0,T]$,
and use the bound on $\Dx J_h$ from the previous step,
to conclude that $\Dx^2 J_h$ and $\Dt J_h$ are uniformly bounded
independent of $h$ in $L^2$ on $[\eps,1]\times[0,T]$.
By the anisotropic Sobolev estimates in Appendix A, we deduce
that $\Dx J_h$ is uniformly bounded independent of $h$ in $L^4$ 
on $[\eps,1]\times[0,T]$, as well.

4. 
Lastly we derive an interior estimate on $\Dx^3 J_h$. We define
\[
\beta(x,t) = -2x + 2\neph + (3-2\neph)\chi_h \,, \quad\mbox{so}\quad 
\Dt\beta = 2(1-\chi_h)\Dx J_h\,.
\]
Then
\begin{equation}
\D_t^2 J_h = x^2\Dx^2\Dt J_h + \beta \Dx\Dt J_h + 2(1-\chi_h)(\Dx J_h)^2,
\end{equation}
and we let $\eta(x)=x-\eps$ so that $\eta(\eps)=0$ and $\eta'=1$, 
\begin{align*}
&\frac{d}{dt}  \int_\eps^1 \frac12 \eta^2(\Dt J_h)^2\,dx 
= \int_\eps^1 \eta^2 (\Dt J_h)(\Dt^2 J_h)\,dx
\\ &\quad = 
\int_\eps^1 \eta^2(\Dt J_h)( \beta \Dx\Dt J_h + 2(1-\chi_h)(\Dx J_h)^2)\,dx
\\ &\qquad\quad
-\int_\eps^1 \eta^2x^2(\Dx\Dt J_h)^2\,dx 
-\int_\eps^1 (\Dt J_h)(\Dx\Dt J_h)\Dx(\eta^2x^2)\,dx
\\& \quad \le -\frac{\eps^2}2 \int_\eps^1 \eta^2(\Dx\Dt J_h)^2\,dx
+ C \int_\eps^1 (\Dt J_h)^2 + (\Dx J_h)^4 \,dx
\end{align*}
Because only know $\Dt J_h$ is continuous for $t>0$, 
we integrate this over $t\in[s,T]$, then over $s\in[0,\tau]$, and use
the bounds from the previous step. We infer that 
$\eta\Dx\Dt J_h$ is uniformly bounded independent of $h$ in $L^2$ on
$[\eps,1]\times[\tau,T]$. Due to \qref{e:Jht} and \qref{d:chih},
we infer that $\Dt(\Dx J_h)$ and $\Dx^2(\Dx J_h)$ are uniformly bounded independent of $h$
in $[\eps+\hat\eps,1-\hat\eps]\times[\tau,T]$,
for any small fixed $\hat\eps>0$ and compact $[\tau,T]\subset(0,\infty)$.

\subsection{Compactness argument.} 
By the anisotropic Sobolev estimates in Appendix A, 
from the bounds on $\Dt J_h$ and $\Dx^2 J_h$ in
step 3 above, we have that $J_h$ is uniformly \holder-continuous
(independent of $h$) on any compact set 
\begin{equation}\label{c:set2}
[\eps,1]\times[0,T] \subset [\eps,1]\times[0,\infty)=\bar Q^\eps\,.
\end{equation}
Also, by the bounds on $\Dt n_h$ and $\Dx^2 n_h$ in
step 2, $n_h$ is uniformly \holder-continuous
on any compact set of the form in \qref{c:set1}.
From this we infer by \qref{d:Jh} for $x<1-2h$ 
the same for $\Dx n_h$.
By step 4, $\Dt n_h=\Dx J_h$ and $\Dx^2n_h$ are uniformly \holder-continuous
on any compact set
\begin{equation}\label{c:set4}
[\eps+\hat\eps,1-\hat\eps]\times[\tau,T]\subset (\eps,1)\times(0,\infty)\,.
\end{equation}

From the Arzela-Ascoli theorem and a diagonalization argument, 
along a subsequence of $h\to0$ we get uniform convergence of: 
$J_h$ to $J_\eps$ in sets of form \qref{c:set2}, 
$n_h$ and $\Dx n_h$ to respective limits $n_\eps$ and $\Dx n_\eps$ 
in sets of form \qref{c:set1},
and $\Dt n_h$ to $\Dt n_\eps$ and $\Dx^2n_h$ to $\Dx^2 n_\eps$
in sets of form \qref{c:set4}, with all limits \holder-continuous 
on the indicated sets.

In the limit, the PDE $\Dt n_\eps=\Dx J_\eps$ holds for $(x,t)\in Q^\eps$, and 
\begin{equation} \label{e:Je2}
J_\eps = x^2 \Dx n_\eps -2xn_\eps + n_\eps^2 \,, \quad (x,t)\in[\eps,1)\times[0,\infty)\ .
\end{equation}
Because of the continuity of $J_\eps$  on the sets in \qref{c:set2}
and $n_\eps$ on the sets in \qref{c:set1}, by regarding \qref{e:Je2}
as an ODE for $n_\eps$ we deduce that $n_\eps$ and 
$\Dx n_\eps$ are continuous on the sets in \qref{c:set2} also
(i.e., up to the boundary $x=1$), 
and both boundary conditions in \qref{e:nep3}--\qref{e:nep4} hold. 

From standard parabolic theory as before, we find that
$n_\eps$ is smooth in $Q^\eps$. 
This concludes the proof of Proposition~\ref{t:regexist}.

\section{Regularity away from the origin}\label{s:interior}

What we seek to do in this section is to prove Theorem~\ref{t:regeg}, providing
sufficient local regularity in the domain  $Q=(0,1]\times(0,\infty)$  
to infer that the solutions $n$ in Theorem~\ref{t.exist} are classical,
with at least the regularity needed for the strict contraction estimate
in Lemma~\ref{contraction}.
For higher  regularity in the interior of $Q$ we will rely on standard parabolic theory
via bootstrap arguments.

The idea is to obtain uniform local bounds on $L^2$ norms
of the solutions $\neph$ of the flux-cutoff problem, the associated
fluxes $J_h$ in \qref{d:Jh}, and certain space-time derivatives $\D^\alpha \neph$, $\D^\beta J_h$.
These bounds will be independent of $h$, $\eps$ and the smoothing parameter $\kappa$. 
The local $L^2$ bounds on these derivatives are inherited by 
$\D^\alpha \nep$, $\D^\beta J_\eps$ in the limit $h\to0$,
then by $\D^\alpha n_\kappa$, $\D^\beta J_\kappa$ 
after taking $\eps\to0$, and then by $\D^\alpha n$, $\D^\beta J$ after taking $\kappa\to0$.
Local \holder-norm bounds for each quantity $v\in\{n,\Dx n,\Dt n,\Dxx n\}$ 
in $Q$ will follow from the local $L^2$ bounds on $\Dt v$ and $\Dxx v$, due to Theorem \ref{t:ani}.  

In order to achieve this, we proceed to first obtain the needed estimates for $n_h$ and $J_h$, 
independent of  $h$, $\eps$, and $\kappa$, and then pass to the limits. 
Select a smooth function $\bar \eta\colon\R\to[0,\infty)$, convex and nondecreasing with 
$0=\bar \eta(0)<\bar \eta(x)\le x $  for $x>0$. 
Weighted energy estimates with weight $\eta(x)=\bar\eta(x-ma)$ will yield 
uniform estimates in $L^2(W_m)$, where the sets $W_m\subset [\epsilon, 1]\times [s, T]$ 
have the form
\begin{equation} \label{cpt1}
W_m =[(m+1)a,1]\times[ms,T] \,,\qquad m=1,2,\ldots,
\end{equation}
for $a,s>0$ arbitrary but fixed independent of  $h$, $\eps$, and $\kappa$.

0. As a preliminary step,  we seek a uniform pointwise bound on $n_h$ 
independent of $h$, $\epsilon$ and $\kappa$, in domains of the form
\begin{equation}\label{dom:etinf}
[\epsilon, 1]\times [\tau, \infty), \qquad \tau>0. 
\end{equation}
From the form of $J_h$ and $J_\epsilon$ it follows that 
\begin{align*}
& n_h(x, t)-n_\eps(x, t)  =n_h(1/2, t)-n_\eps(1/2, t) +\int_{1/2}^x \frac{1}{y^2}(J_h-J_\eps)\,dy \\
& \quad + \int_{1/2}^x \left[ \frac{2}{y}(n_h-n_\eps) -\frac{1}{y^2}(n_h^2-n_\eps^2)\right]dy  
 \quad + \int_{1-2h}^x  \frac{1}{y^2}(n_h^2-3n_h)\chi_h\,dy.
\end{align*}
Using the uniform convergence of $n_h$ to $n_\epsilon$ in $[\epsilon, 1-\hat \epsilon]\times [0, T]$
(proven previously), and of $J_h$ to $J_\epsilon$ in $[\epsilon, 1]\times [0, T]$, 
as well as the bounds on $n_h$ in Lemma \ref{l:max-h} 
and on $n_\epsilon \leq S(x, t)$ in Lemma  \ref{l:uss},  
we obtain the uniform convergence of $n_h$ toward $n_\eps$ as $h\to 0$.   
Therefore, we get the following uniform pointwise bound independent of $h$, $\eps$, and  $\kappa$:
For any $\tau>0$, for sufficiently small $h>0$ we have
\begin{equation} \label{bf}
0 < n_h(x,t) \le M_\tau=\max_{x\in [0, 1]} S(x, \tau) +1,  \qquad (x, t)\in [\eps,1]\times[\tau, \infty).
\end{equation}
(Here and below, the required smallness of $h$ depends on $\kappa$, 
because the bound in Lemma \ref{l:uss} depends on $\kappa$. But we will not mention this further.)  

1. Next we proceed to obtain bounds using weighted energy estimates.
The weighted energy estimate with $\eta(x)=\bar \eta(x-a)$  is 
\begin{align*}
&\frac{d}{dt}  \int_\eps^1 \frac12 \eta^2  n_h^2\,dx  
= \int_\eps^1 \eta^2 n_h\,\Dx J_h\,dx 
=  - \int_\eps^1 (\eta^2 \Dx n_h+2\eta \eta' n_h )J_h\,dx
\\
&\quad =  
-\int_a^1
(\eta^2 \Dx n_h+2\eta \eta' n_h )(x^2\Dx\neph -2x\neph +  \neph^2 + (3\neph-\neph^2)\chi_h)\,dx
\\ &\quad \le -\frac{1}2 
\int_a^1 (x\eta\, \Dx n_h)^2\,dx + C \int_a^1 (n_h+n_h^2)^2\,dx. 
\end{align*}
By integration over $t\in [s, T]$ and using (\ref{bf}),  we infer that
\begin{align} \notag 
  &  \int_s^T\int_{2a}^1 (x\eta\,\Dx n_h)^2dx\,d t \\ 
& \qquad \leq 
 \int_\eps^1\eta^2 n_h^2 (x, s)\,dx +C \int_s^T \int_a^1(n_h+n_h^2)^2 \,dx\,dt 
 \notag \\  
&\qquad \leq C_s\,,
  \label{c3} 
  \end{align}
where $C_s$ may depend on $s$ (and $T$, but we suppress this dependence),
but is independent of $h$, $\epsilon$, $\kappa$.  
Because $\eta(2a)>0$, we conclude that $\Dx n_h$,
hence also $J_h$, 
is uniformly bounded independent of $h, \eps$ and $\kappa$ in $L^2(W_1)$
(with a bound that depends on $a$ and $s$).

2.  The cut-off flux $J_h$  satisfies
\begin{equation}\label{e:Jt}
\Dt J_h=x^2 \Dx^2 J_h +\beta \Dx J_h \,,  
\end{equation}
with boundary condition $J_h(1,t)=0$ for $t>0$, where 
\[
\beta(x,t) = -2x + 2\neph + (3-2\neph)\chi_h \,.
\]
Multiply by $\eta^2 J_h$ with $\eta(x)=\bar \eta(x-2a)$, integrate by parts, and use the
inequality $uv\le \frac14 u^2+v^2$ to obtain
\begin{eqnarray*}
&& \frac{d}{dt}\int_\eps^1 \frac12\eta^2 J_h^2 \,dx 
\\&& \qquad
=
-\int_\eps^1 (x\eta\,\Dx J_h)^2\,dx  +
\int_\eps^1 
J_h (\Dx J_h)(\eta^2\beta - \Dx(x^2\eta^2))\,dx
\\
&& \qquad \le 
-\int_\eps^1 (x\eta\,\Dx J_h)^2\,dx  +
\int_\eps^1 |J_h\Dx J_h|\cdot 2x\eta C_s\,dx 
\\
&& \qquad \le 
-\frac12 \int_\eps^1 (x\eta\,\Dx J_h)^2\,dx 
+ C_s \int_{2a}^1 
 J_h^2 \,dx \,.
\end{eqnarray*}
Integrating over $t\in[\tau,T]$ first, then averaging over $\tau \in [s, 2s]$, we obtain 
\begin{align}\notag 
&\int_{2s}^T\int_{3a}^1 (x\eta\, \Dx J_h)^2dx\,dt  \\ 
& \qquad \leq 
 \frac{1}{s}\int_s^{2s} \int_{2a}^1\eta^2 J_h^2(x, \tau)dx\,d\tau  + C_s \int_{s}^T  \int_{2a}^1 
 J_h^2 \,dx\,dt 
 \notag\\
 &\qquad \leq C(a, s). 
 \label{es2}
\end{align}
Here we have used $|J_h|^2\leq C(x^2|\partial_x n_h|^2 +n_h^2 +n_h^4)$, (\ref{c3}) 
and the upper bound on $n_h$ in (\ref{bf}).  We conclude that
$\Dx J_h$ is uniformly bounded in $L^2(W_2)$, independent of $h$, $\eps$, $\kappa$. 
Thus $\Dt n_h$ (but not $\Dx^2 n_h$) 
is uniformly bounded in the same $L^2$ sense. 

3. Let us write $n_1 = \Dt n_h=\Dx J_h$.
Then for $t>0$,  
\begin{equation}\label{c7}
\Dt n_1 = \Dx J_1, \qquad J_1(1,t)=0,
\end{equation}
where
\begin{equation}\label{d:J1+}
J_1=\Dt J_h = x^2\Dx n_1 + \beta n_1. 
\end{equation}
Note that the validity of the zero-flux condition $J_1(1, t)=0$ is implied by the H\"{o}lder continuity of $J_1$. 
To see this is valid, set $v=J_h-n_h^2(1-\chi_h)$. From (\ref{e:Jht}) it follows that $v$ solves 
$$
\partial_t v- x^2 \partial_x^2 v =F,
$$ 
subject to homogeneous boundary conditions, where the source term   
$$
F=\partial_x J_h (-2x +3\chi_h)+x^2\partial_x^2(n_h^2(1-\chi_h)).
$$
From the results in Appendix~\ref{s:regexist},
$F$ is locally H\"{o}lder-continuous on $[\eps, 1]\times (0, \infty)$.  
Hence, $J_1= \partial_t v+2n_h\partial_t n_h(1-\chi_h)$ is continuous up to $x=1$. 
 
Multiply (\ref{c7}) by $\eta^2 n_1$ with $\eta(x)=\bar \eta(x-3a)$, and integrate in $x$ over $[\epsilon, 1]$ to obtain
\begin{align*}
& \frac{d}{dt}\int_\eps^1 \frac12\eta^2 n_1^2 \,dx 
= - \int_\eps^1 (\eta^2\Dx n_1+2\eta\eta' n_1) (x^2\Dx n_1+ \beta n_1)\,dx
\\
&\qquad \le
- \int_\eps^1 (x\eta\,\Dx n_1)^2 \,dx
+ \int_\eps^1 \Bigl(
(2x^2\eta \eta'+|\beta|\eta^2 )|n_1\Dx n_1| + 
2\eta\eta' |\beta| n_1^2 \Bigr)\,dx
\\
&\qquad \le
-\frac12 \int_\eps^1 (x\eta\,\Dx n_1)^2 \,dx
+ C_s \int_{3a}^1
  n_1^2 \,dx \,.
\end{align*}
Integrating over $t\in[\tau,T]$ first, then over $\tau \in [2s, 3s]$, we obtain
\begin{align}\label{es3} \notag 
&\int_{3s}^T\int_{4a}^1 (x\eta \, \Dx n_1 )^2dx\,d\tau 
\\&\qquad
 \leq  \frac{1}{s}\int_{2s}^{3s} \int_{3a}^1\eta^2 n_1^2(x, \tau)\,dx\,d\tau 
+ C_s\int_{2s}^T \int_{3a}^1 n_1^2 \,dx\,d\tau 
\notag\\
&\qquad \leq C(a, s),
\end{align}
where the bound on $n_1=\Dx J_h$ in (\ref{es2}) from step 2 has been used.  
We conclude that $\Dx^2 J_h$ and $\Dt J_h$ (by \eqref{e:Jt}) are uniformly bounded
independent of $h, \eps$, and $\kappa$ in $L^2$ on $W_3$,  
hence $J_h$ is uniformly H\"older continuous on $W_3$.

4. Next we compute $\Dt J_1$ to complete the estimates for classical solutions. 
Differentiating \qref{d:J1+}  with respect to $t$ we find that for $t>0$,
\begin{equation}\label{e:J1t}
\Dt J_1 = x^2\Dx^2 J_1 + \beta \Dx J_1 + \Dt \beta\, n_1, \quad J_1(1,t) = 0.
\end{equation}
Recall that $|\beta|\leq C_s$ and note 
$\Dt\beta = 2(1-\chi_h)\Dx J_h,$
hence $|\partial_t \beta|\leq 2|n_1|$.  
Multiply by $\eta^2 J_1$ with $\eta(x)=\bar \eta(x-4a)$,  and integrate by parts to find
\begin{align*}
& \frac{d}{dt}\int_\eps^1 \frac12\eta^2 J_1^2 \,dx 
+\int_\eps^1 (x\eta\,\Dx J_1)^2\,dx  
\\
&\qquad =
\int_\eps^1 
(\beta \eta^2 - \Dx(x^2\eta^2))J_1(\Dx J_1) +\eta^2  J_1\Dt \beta\,n_1\,dx
\\
&\qquad  \le \frac12\int_\eps^1 (x\eta\,\Dx J_1)^2\,dx + 
C_s \left(\int_{4a}^1 J_1^2\,dx + \int_{4a}^1 n_1^4\,dx\,\right).
\end{align*}
Integrating over $t\in[\tau,T]$ first, then averaging over $\tau \in [3s, 4s]$, we obtain
\begin{align*} \notag 
\int_{4s}^T\int_{5a}^1 (x\eta\,  \Dx J_1 )^2dx\,dt 
&\leq 
 \frac{1}{s}\int_{3s}^{4s} \int_{4a}^1 \eta^2 J_1^2(x, \tau) \,dx\,d\tau \\\notag 
 & \quad +  C_s \int_{3s}^T 
 \left(  \int_{4a}^1 J_1^2\,dx + \int_{4a}^1 \ n_1^4\,dx\right)dt \\ \notag  
 & \leq C_s\int_{3s}^T \int_{4a}^1 (|\partial_x n_1|^2+|\partial_x J_h|^2+|n_1|^4)\,dx\,dt. 
\end{align*}
The first two terms are bounded using the bounds from the previous steps, (\ref{es2}) and (\ref{es3}).  
Note also that $n_1=\Dx J_h$ is in $L^4(W_3)$ due to an anisotropic embedding theorem. Hence
 \begin{align} \label{es4}
\int_{4s}^T\int_{5a}^1 (x\eta\,  \Dx J_1 )^2dx\,d\tau \leq C(a, s). 
\end{align}
We can conclude that $\Dx J_1$ ($=\Dx\Dt J_h=\Dt\Dx J_h=\Dt^2n_h$) is bounded in $L^2(W_4)$
independent of $h, \eps$, and $\kappa$.

5. After taking the limit $h\to0$ along a suitable subsequence, 
we conclude from steps 1 and 2 that $\Dt\nep=\Dx J_\eps$ is uniformly bounded in $L^2(W_2)$,
hence the same is true of $\Dxx \nep$ due to the form of $J_\eps$ in \eqref{d.Jep}.
By Theorem~\ref{t:ani}, $\nep$ is uniformly \holder-continuous on $W_2$, 
independent of $\eps$ and $\kappa$.

Next we conclude from step 3 that $J_\eps$ is uniformly \holder-continuous on $W_3$,
and the same is true of $\Dx \nep$ by \eqref{d.Jep}. 

From step 4 we then conclude $\Dt\Dx J_\eps$ is uniformly bounded in $L^2(W_4)$
and by differentiating \eqref{d.Jep}  we conclude the same for $\Dx^3 J_\eps$. 
Therefore $\Dx J_\eps=\Dt \nep$ is uniformly \holder-continuous on $W_4$,
and the same holds for $\Dx^2\nep$. 

After taking the limits $\epsilon \to 0$, and finally $\kappa \to 0$, 
these estimates ensure that the weak solution $n$ of Theorem~\ref{t.exist}
is a classical solution in $Q=(0,1]\times(0,\infty)$, with the local \holder\ regularity
indicated in Theorem~\ref{t:regeg}.

\section*{Acknowledgments.}  We want to thank the IPAM for the hospitality and support during our visit in May-June 2009, when this work was initiated.
This material is based upon work supported by the National
Science Foundation under 
the NSF Research Network Grant no.\ RNMS11-07444, \ RNMS11-07291(KI-Net),
and grants   DMS 0907963 and DMS 1312636 (HL),
DMS 0905723, DMS 1211161 and DMS 1515400 (RLP).
RLP was partially supported by the Center for Nonlinear Analysis (CNA)
under National Science Foundation PIRE Grant no.\ OISE-0967140.


\bibliographystyle{siam}
\bibliography{photon}

\end{document}